\documentclass[11pt]{amsart}

\usepackage[
paper=a4paper,
text={147mm,230mm},centering
]{geometry}

\iftrue
\makeatletter
\def\@settitle{%
  \vspace*{-20pt}
  \begin{flushleft}%
    \baselineskip14\p@\relax
    \normalfont\bfseries\LARGE
    \@title
  \end{flushleft}%
}
\def\@setauthors{%
  \begingroup
  \def\thanks{\protect\thanks@warning}%
  \trivlist
  \large \@topsep30\p@\relax
  \advance\@topsep by -\baselineskip
  \item\relax
  \author@andify\authors
  \def\\{\protect\linebreak}%
  \authors
  \ifx\@empty\contribs
  \else
    ,\penalty-3 \space \@setcontribs
    \@closetoccontribs
  \fi
  \normalfont
  \@setaddresses
  \endtrivlist
  \endgroup
}
\def\@setaddresses{\par
  \nobreak \begingroup\raggedright
  \small
  \def\author##1{\nobreak\addvspace\smallskipamount}%
  \def\\{\unskip, \ignorespaces}%
  \interlinepenalty\@M
  \def\address##1##2{\begingroup
    \par\addvspace\bigskipamount\noindent
    \@ifnotempty{##1}{(\ignorespaces##1\unskip) }%
    {\ignorespaces##2}\par\endgroup}%
  \def\curraddr##1##2{\begingroup
    \@ifnotempty{##2}{\nobreak\noindent\curraddrname
      \@ifnotempty{##1}{, \ignorespaces##1\unskip}\/:\space
      ##2\par}\endgroup}%
  \def\email##1##2{\begingroup
    \@ifnotempty{##2}{\smallskip\nobreak\noindent E-mail address%
      \@ifnotempty{##1}{, \ignorespaces##1\unskip}\/:\space
      \ttfamily##2\par}\endgroup}%
  \def\urladdr##1##2{\begingroup
    \def~{\char`\~}%
    \@ifnotempty{##2}{\nobreak\noindent\urladdrname
      \@ifnotempty{##1}{, \ignorespaces##1\unskip}\/:\space
      \ttfamily##2\par}\endgroup}%
  \addresses
  \endgroup
  \global\let\addresses=\@empty
}
\def\@setabstracta{%
    \ifvoid\abstractbox
  \else
    \skip@25\p@ \advance\skip@-\lastskip
    \advance\skip@-\baselineskip \vskip\skip@
    \box\abstractbox
    \prevdepth\z@ 
    \vskip-10pt
  \fi
}
\renewenvironment{abstract}{%
  \ifx\maketitle\relax
    \ClassWarning{\@classname}{Abstract should precede
      \protect\maketitle\space in AMS document classes; reported}%
  \fi
  \global\setbox\abstractbox=\vtop \bgroup
    \normalfont\small
    \list{}{\labelwidth\z@
      \leftmargin0pc \rightmargin\leftmargin
      \listparindent\normalparindent \itemindent\z@
      \parsep\z@ \@plus\p@
      
    }%
    \item[\hskip\labelsep\bfseries\abstractname.]%
}{%
  \endlist\egroup
  \ifx\@setabstract\relax \@setabstracta \fi
}
\def\section{\@startsection{section}{1}%
  \z@{-1.2\linespacing\@plus-.5\linespacing}{.8\linespacing}%
  {\normalfont\bfseries\large}}
\def\subsection{\@startsection{subsection}{2}%
  \z@{-.8\linespacing\@plus-.3\linespacing}{.3\linespacing\@plus.2\linespacing}%
  {\normalfont\bfseries}}
\def\subsubsection{\@startsection{subsubsection}{3}%
  \z@{.7\linespacing\@plus.1\linespacing}{-1.5ex}%
  {\normalfont\itshape}}
\def\@secnumfont{\bfseries}
\makeatother
\fi 

\usepackage{amssymb}
\usepackage{mathrsfs}
\usepackage[all]{xy}
\usepackage{graphicx,color,float}
\usepackage[bookmarks]{hyperref}
\usepackage{pinlabel}
\usepackage[textwidth=1.3in,color=yellow]{todonotes}
\usepackage{amsmath}
\usepackage{pb-diagram,pb-xy}
\usetikzlibrary{calc}
\textwidth=\paperwidth \advance\textwidth by-2\oddsidemargin
\advance\oddsidemargin by-1in
\evensidemargin=\oddsidemargin
\calclayout

\def\N{\mathbb{N}}
\def\Z{\mathbb{Z}}
\def\Q{\mathbb{Q}}
\def\R{\mathbb{R}}
\def\C{\mathbb{C}}
\def\P{\mathbb{P}}

\def\pa{\partial}

\def\Ker{\operatorname{Ker}}

\def\Im{\operatorname{Im}}
\def\Hom{\operatorname{Hom}}

\def\+{\oplus}

\def\u{\textbf{\textup{u}}}
\def\v{\textbf{\textup{v}}}
\def\x{\textbf{\textup{x}}}
\def\m{\textbf{\textup{m}}}

\def\z{\textbf{\textup{z}}}
\def\y{\textbf{\textup{y}}}

\def\T{\textup{Trop}}
\def\Trop{\textup{Trop}}
\def\trop{\textup{trop}}
\def\ev{\textup{ev}}

\def\ulp{\textup{(}}
\def\urp{\textup{)}\,}
\def\scr{\mathscr}
\def\mcal{\mathcal}

\definecolor{darkgreen}{rgb}{0.0, 0.2, 0.13}
\definecolor{green2}{rgb}{0.0, 0.5, 0.0}

\theoremstyle{plain}
\newtheorem{theorem}{Theorem}[section]
\newtheorem{thmx}{Theorem}
\newtheorem{proposition}[theorem]{Proposition}
\newtheorem{corollary}[theorem]{Corollary}
\newtheorem*{corollaryA}{Corollary}
\newtheorem{lemma}[theorem]{Lemma}
\newtheorem{sublemma}[theorem]{Sublemma}
\theoremstyle{definition}
\newtheorem{definition}[theorem]{Definition}
\newtheorem*{definitionA}{Definition}
\newtheorem*{question}{Question}
\newtheorem{example}[theorem]{Example}
\newtheorem{remark}[theorem]{Remark}

\def\to{\mathchoice{\longrightarrow}{\rightarrow}{\rightarrow}{\rightarrow}}
\makeatletter
\newcommand{\shortxra}[2][]{\ext@arrow 0359\rightarrowfill@{#1}{#2}}
\def\longrightarrowfill@{\arrowfill@\relbar\relbar\longrightarrow}
\newcommand{\longxra}[2][]{\ext@arrow 0359\longrightarrowfill@{#1}{#2}}

\makeatother
\numberwithin{equation}{section}

\begin{document}

\title[Non-displaceable toric fibers via tropicalizations]
{Non-displaceable toric fibers on compact toric manifolds via tropicalizations}

\author{Yoosik Kim, Jaeho Lee}

\address{Department of Mathematics\\
University of Wisconsin at Madison\\
WI \\
USA}

\email{ykim@math.wisc.edu, jlee@math.wisc.edu}


\begin{abstract}
We give a combinatorial way to locate non-displaceable Lagrangian toric fibers on compact toric manifolds. By taking the intersection of certain tropicalizations coming from combinatorial data of a moment polytope, we locate all strongly bulk-balanced fibers introduced in \cite{FOOOSurv}. As an application, we show that every bulk-balanced fiber defined in \cite{FOOOToric2} is strongly bulk-balanced. Thus, the method indeed detects the positions of all non-displaceable fibers that can be detected by Lagrangian Floer theory developed in \cite{FOOOToric1} and \cite{FOOOToric2}.
\end{abstract}

\maketitle

\section{Introduction} \label{section:introduction}

Finding rigid submanifolds on a symplectic manifold has been one of fundamental questions in symplectic topology. Among many notable symplectic manifolds, symplectic toric manifolds have attracted special attention of symplectic topologists because they have provided interesting examples but are nonetheless accessible. The rigidity problem that we are concerned with is which Lagrangian toric fibers are non-displaceable by a Hamiltonian diffeomorphism on a compact symplectic toric manifold. There are several approaches to this problem. One approach taken by Cho-Oh \cite{ChoOh}, Fukaya-Oh-Ohta-Ono \cite{FOOOToric1}, \cite{FOOOToric2} and Woodward \cite{W} for instance is searching which toric fibers have a non-trivial cohomology arising from certain Lagrangian Floer theories. Another approach developed by Entov-Polterovich \cite{EP} is using a quasi-state to show the existence of non-displaceable toric fibers. On the other hand, McDuff \cite{Mc} and Abreu-Borman-McDuff \cite{ABM} introduced the method of probes to find displaceable toric fibers. 

In the firstly mentioned approach in \cite{FOOOToric1}, with the aid of toric structure, the potential function restricted to the 1-cochains of fibers, which is so called the Landau-Ginzburg superpotential in the physics literature, can be expressed as a Laurent power series in terms of variables from a basis of the dual lattice. In general, it consists of two parts: the leading order potential function (also called the Hori-Vafa potential) written from the equations of the supporting hyperplanes of a moment polytope, and the correction terms coming from the contribution of holomorphic spheres. The potential function plays a crucial role in Lagrangian Floer theory on toric manifolds because the differential of the Floer complex is determined by the partial derivatives of the potential function so that detecting a position with a non-vanishing Floer cohomology is equivalent to finding a position admitting a critical point of the potential function. 

In \cite{FOOOToric2}, a bulk deformation by components of the toric divisor was used to deform the potential function (more essentially the underlying $A_\infty$-algebra), which allows us to conclude non-displaceability of toric fibers once the \emph{leading term equation} ~\eqref{leadingterm}, a part of the leading order potential function, has a solution on $(\C^*)^n$. Moreover, using bulk-deformations with slightly extended coefficients, each coefficient in the leading term equation can be independently controlled, which allows us to have some flexibility when solving a system of equations. Such a deformed leading term equation is called a \emph{generalized leading term equation} ~\eqref{genleadingterm}. In this sense, our problem is reduced to solving a certain system of equations.

A relevant discussion on solvability of a system of equations in tropical geometry can be found in the work of Osserman-Payne \cite{OP}. It includes the intersection of \emph{tropicalizations} (Definition~\ref{tropicalizationofequation}) lifts to a solution of the system over a valued field whenever they intersect \emph{properly} (See Definition~\ref{intersectproperly}), and it generalizes the lifting results of Bogart-Jensen-Speyer-Sturmfels-Thomas \cite{BJS+} dealing with the case where the tropicalizations intersect transversally. The result can be applied to some extent to detect non-displaceable toric fibers when tropicalizations intersect properly (See Corollary~\ref{Tropicaliftingbalanced} and Example~\ref{Onepointblowup1}). 

In general, however, the intersection does not lift to a solution if the tropicalizations of a system intersect improperly. Thus, in the improperly intersecting case, the toric fibers over the intersection are not necessarily non-displaceable. Nevertheless, we might have hope that the intersection lifts to non-displaceable fibers because of the flexibility from bulk-deformations. As in Figure~\ref{TwopointblowupFig0}, the tropicalizations of the components of the gradient of the potential function in the two-point blowup of $\C\P^2$ appeared in \cite{FOOOToric2} exactly intersect at the positions on which the fibers are non-displaceable even though the tropicalizations intersect improperly. 

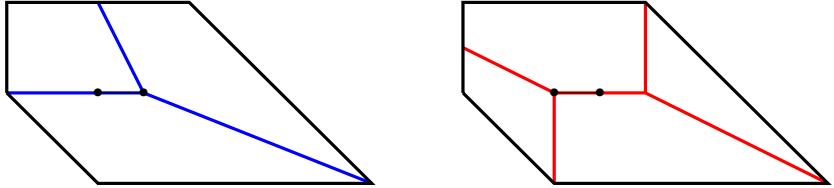
\begin{figure}[h]
\centering
\begin{tikzpicture}
 \def\size{1.2}
 \def\trans{6}

 \coordinate (A) at (0, \size);
 \coordinate (B) at (\size, 0);
 \coordinate (C) at (4 * \size, 0);
 \coordinate (D) at (2 * \size, 2 *\size);
 \coordinate (E) at (0, 2 * \size);
 \coordinate (F) at (1.5 * \size , \size);
 \coordinate (G) at (\size, \size);
 \coordinate (H) at (2 * \size, \size);
 \coordinate (I) at (0, 1.5 * \size);
 \coordinate (J) at (\size, 2 * \size);  
 \coordinate (K) at (0, \size);
 \coordinate (L) at (4 * \size, 0);
 \coordinate (M) at (3 * \size, \size);
 \coordinate (N) at (1/2 * \size, 1/2 * \size);

 \draw [very thick, color=red] (G) -- (H) ;
 \draw [very thick, color=red] (H) -- (C) ;
 \draw [very thick, color=red] (H) -- (D) ;
 \draw [very thick, color=red] (G) -- (B) ;
 \draw [very thick, color=red] (G) -- (I) ;
 \draw [thin, color=black] (G) -- (F) ;
 \node [black] at (\size, \size) {\tiny$\bullet$}; 
 \node [black] at (1.5* \size, \size) {\tiny$\bullet$}; 
 \draw [very thick] (A) -- (B) -- (C) -- (D)-- (E)-- (A);

 \coordinate (A) at (0 -\trans, \size);
 \coordinate (B) at (\size - \trans, 0);
 \coordinate (C) at (4 * \size -\trans, 0);
 \coordinate (D) at (2 * \size - \trans, 2 *\size);
 \coordinate (E) at (0- \trans, 2 * \size);
 \coordinate (F) at (1.5 * \size -\trans, \size);
 \coordinate (G) at (\size -\trans, \size);
 \coordinate (H) at (2 * \size - \trans, \size);
 \coordinate (I) at (0 -\trans, 1.5 * \size);
 \coordinate (J) at (\size -\trans, 2 * \size);  
 \coordinate (K) at (0 -\trans, \size);
 \coordinate (L) at (4 * \size-\trans, 0);
 \coordinate (M) at (3 * \size-\trans, \size);
 \coordinate (N) at (1/2 * \size-\trans, 1/2 * \size);

 \draw [very thick, color=blue] (F) -- (J) ;
 \draw [very thick, color=blue] (F) -- (K) ;
 \draw [very thick, color=blue] (F) -- (L) ;
 \draw [thin, color=black] (G) -- (F) ;
 \node [black] at (\size-\trans, \size) {\tiny$\bullet$}; 
 \node [black] at (1.5* \size-\trans, \size) {\tiny$\bullet$}; 
 \draw [very thick] (A) -- (B) -- (C) -- (D)-- (E)-- (A);

\end{tikzpicture}
\caption{The intersection of tropicalizations in the two-point blowup of $\C\P^2$.}\label{TwopointblowupFig0}
\end{figure}

It turns out however the tropicalizations coming from the gradient of the potential function is insufficient to filter out the positions of displaceable fibers (See Example~\ref{Twopointblowup1}). It is essentially because the tropicalizations generally fails to encode information of all components in the leading term equation. To get rid of irrelevant fibers or to encode enough information, we consider the tropicalization of the logarithmic derivative of the leading order potential function with respect to the direction of a lattice point. In a moment polytope, this tropicalization is combinatorially realized as the tropicalization relative to a designated lattice direction as follows. 

\begin{definitionA}[Definition~\ref{tropicalizationexcepting}]\label{tropicalizationexcepting}
Let $P \subset \R^n$ be the polytope whose boundary is given by the linear equations $\{ l_j (\u):= \langle \u, \v_j \rangle + \lambda_j : 1 \leq j \leq m \}$ satisfying the requirements in \eqref{descriptionofpolytope} where $\u = (u_1, \cdots, u_n)$ is the standard coordinate system on $\R^n$. The \emph{tropicalization of} $P$ \emph{relative to} $\m \in \Z^n$ is denoted by $\Trop \left(P \, , \m \right)$ and is defined to be the non-differentiable locus of the piecewise-linear function $\trop^{P, \m}$ given by 
$$
\trop^{P, \m} : \R^n \to \R, \quad \u \mapsto \min \left\{ l_j (\u) : \langle \m, \v_j \rangle \neq 0 \right\}.
$$
\end{definitionA}

To present the main Theorem effectively, we introduce one more definition. 
\begin{definitionA}
Let $P \subset \R^n$ be the polytope given by the intersection of the half spaces $\{ \langle \u, \v_j \rangle + \lambda_j \geq 0 : 1 \leq j \leq m \}$ satisfying the requirements in \eqref{descriptionofpolytope}. A subspace of the vector space $\R^n$ is called a \emph{primary subspace for} $P$ if it is a subspace in $\R^n$ of codimension one generated by facet normal vectors $\v_j$'s of $P$. 
\end{definitionA}

We can now state the main result of this paper. 

\begin{thmx}[Theorem~\ref{maintheroem1}]\label{THEOREMA}
Let $X$ be the compact symplectic toric manifold determined by a moment polytope $P$. For a point $\u$ in the interior of $P$, the followings are equivalent:
\begin{enumerate}
\item The toric fiber over the point $\u$ is strongly bulk-balanced. Namely, a generalized leading term equation admits a solution on $(\C^*)^n$ (See Definition~\ref{stronglybulkbalanced}).
\item The point $\u$ is in the intersection of tropicalizations $\Trop(P, \m)$ relative to $\m$ over all lattice points $\m$. 
\item The point $\u$ is in the intersection of tropicalizations $\Trop(P, \m)$ relative to $\m$ over all primitive lattice points $\m$ orthogonal to a primary subspace for $P$. 
\end{enumerate}
\end{thmx}

As a corollary, we have 
\begin{corollaryA}[Corollary~\ref{maincorollary}]
If an interior point $\u$ of $P$ is contained in the intersection of $\, \Trop( P, \m)$ over all primitive lattice points perpendicular to a primary subspace for $P$, then the toric fiber $L(\u)$ over $\u$ is non-displaceable. 
\end{corollaryA}

We notice that the number of primary subspaces for $P$ is bounded by ${m \choose n-1} = \frac{m!}{(n-1)! \cdot (m-n+1)!}$ and hence finite. Moreover, since $\Trop(P, \m) = \Trop(P, -\m)$, it is sufficient to take one primitive vector perpendicular to each $(n-1)$-dimensional subspace. Thus, by considering at most ${m \choose n-1}$ tropicalizations, we can detect the positions of all strongly bulk-balanced fibers.

In \cite{FOOOToric2}, in order to include more non-displaceable toric fibers by Floer theory that they developed, Fukaya-Oh-Ohta-Ono introduced a notion of \emph{bulk-balanced} fibers (See Definition~\ref{bulkbalanced}). Roughly speaking, the notion relies on the fact that the limit of non-displaceable fibers is also non-displaceble. Since tropicalizations incorporate a limit process, it is reasonable to expect that the intersection of the tropicalizations in Theorem~\ref{THEOREMA} should also detect bulk-balanced fibers. As an application of Theorem~\ref{THEOREMA}, we derive

\begin{thmx}[Theorem~\ref{maintheorem2}]\label{THEOREMB}
Let $X$ be the compact symplectic toric manifold determined by a moment polytope $P$. If the toric fiber over a point $\u$ in the interior of $P$ is bulk-balanced, then it is strongly bulk-balanced. 
\end{thmx}

Theorem~\ref{THEOREMB} indeed implies that the intersection of the tropicalizations in Theorem~\ref{THEOREMA} is able to locate all non-displaceable toric fibers that can be detected by current technology of Lagrangian Floer theory (as far as the authors know) in a compact symplectic toric manifold. Along this line, one open question asked in \cite{FOOOToric2} can be rephrased as follows: 

\begin{question}
If $\u \in \text{Int}(P)$ is not in the intersection of the tropicalizations in Theorem~\ref{THEOREMA}, is the toric fiber over $\u$ is displaceable?\footnote{In some cases, extended probes in \cite{ABM} give a complete answer for this question. However, there exists a symplectic toric surface containing a toric fiber that is not on the intersection of tropicalizations and is not displaceable by extended probes (See \cite{ABM}).} 
\end{question}

Because our criterion uses results from \cite{FOOOToric1} and \cite{FOOOToric2}, the authors should admit that it does not detect any extra non-displaceable fibers other than those detected by methods therein. Nevertheless, the criterion has some advantages: Firstly, the method is global so that systems of equations depending on the positions are not necessarily considered. Also, it fully includes bulk-balanced fibers so we do not need to consider any kinds of perturbation of polytopes causing a sequence of non-displaceable fibers. Finally, it converts a problem solving systems of equations into a problem of linear equations via tropical intersection theory, which is expected to be much simpler. 

The rest of the paper is organized as follows. In Section~\ref{LFT on toric}, we recall derivation and roles of the potential function on compact toric manifolds following \cite{FOOOToric1}. Section~\ref{Tropicalizations} discusses when and how lifting results in \cite{OP} deduce non-displaceability of toric fibers. Section~\ref{examples} presents specific examples illustrating how the results in Section~\ref{Tropicalizations} and Theorem~\ref{THEOREMA} are applied to detect non-displaceable toric fibers. In Section~\ref{section:Bulk}, we review bulk-deformations and their effects on the potential function in \cite{FOOOToric2}. Finally, proofs of Theorem~\ref{THEOREMA} and Theorem~\ref{THEOREMB} are given in Section~\ref{ProofA} and Section~\ref{ProofB} respectively. 

\subsection*{Acknowledgements}
We would like to thank our advisor Yong-Geun Oh for his continuous support and invaluable discussions. This project was carried out during our stay at Institute for Basic Science - Center for Geometry and Physics (IBS-CGP).  We would like to thank IBS-CGP for the excellent environment and support. This work was partially supported by IBS project (IBS-R003-D1) during our visirt of IBS-CGP.
\bigskip

\section{Lagrangian Floer theory on toric manifolds}\label{LFT on toric}

In this section, we review Lagrangian Floer theory on compact toric symplectic manifolds developed by Fukaya-Oh-Ohta-Ono, focusing on derivation and roles of the potential function. For complete details, the reader is referred to \cite{FOOOToric1}. 

Throughout this section, let $X$ denote the $2n$-dimensional compact symplectic manifold determined by a moment polytope $P$, the image of a moment map $\pi : X \to M_\R$. A choice of a basis for $M$ gives us the identifications $M_\R \simeq \R^n$ and $N_\R \simeq \R^n$. Under the identifications, the moment polytope $P$ can be expressed as an intersection of $m$ half-spaces in $\R^n$:
\begin{equation}\label{descriptionofpolytope}
\displaystyle P = \bigcap_{j=1}^m \left\{ \u \in \R^n : \langle \u, \v_j \rangle \geq \lambda_j \right\}.
\end{equation}
To achieve the uniqueness of this expression, we additionally require the following conditions on the expression \eqref{descriptionofpolytope}:
\begin{enumerate}
\item The intersection is assumed to be not redundant so that each of half-spaces must share the boundary with a different facet of the polytope $P$. 
\item An inward vector $\v_j \in N$ is supposed to be primitive, which means that the nonzero components of $\v_j$ have the greatest common divisor $1$. 
\end{enumerate}
We denote by $L(\u)$ the toric fiber over $\u \in \text{Int}(P)$ and recall that $L(\u)$ is Lagrangian with respect to the symplectic form $\omega_P$ on $X$ given by $P$.

\subsection{Novikov rings}

We introduce rings that we are working over. The Novikov field $\Lambda$ over $\C$ is the field
\begin{equation}\label{Novikovfield}
\Lambda := \left\{ \sum_{j=1}^{\infty} a_{j} T^{\lambda_j}  \, \bigg{|} \, a_j \in \C, \lambda_j \in \R, \lim_{j \rightarrow \infty} \lambda_j = \infty \right\}
\end{equation}
together with the valuation $\frak{v}_T : \Lambda \backslash \{ 0 \} \to \R$ given by
$$
\frak{v}_T \left(\sum_{j=1}^{\infty} a_{j} T^{\lambda_j}  \right) := \inf_{j} \, \{ \lambda_j : a_j \neq 0 \}.
$$
Here, $T$ is a formal parameter storing data of symplectic areas. This field is known to be algebraically closed (Lemma A.1 in \cite{FOOOToric1}). We will also consider subrings of $\Lambda$ given as follows:
\begin{align*}
\Lambda_0 &:= \frak{v}_T^{-1}[0, \infty) \cup \{ 0 \} = \left\{ \sum_{i=1}^{\infty} a_{i} T^{\lambda_i} \in \Lambda  \, \bigg{|} \, a_i \in \C, \lambda_i \in \R, \lambda_i \geq 0, \lim_{i \rightarrow \infty} \lambda_i = \infty \right\} \\
\Lambda_+ &:= \frak{v}_T^{-1}(0, \infty) \cup \{ 0 \} = \left\{ \sum_{i=1}^{\infty} a_{i} T^{\lambda_i} \in \Lambda  \, \bigg{|} \, a_i \in \C, \lambda_i \in \R, \lambda_i > 0, \lim_{i \rightarrow \infty} \lambda_i = \infty \right\}. \,
\end{align*}
As analogues of the unitary group $U(1)$ and the algebraic torus $\C^*$, we respectively put
\begin{align*}
\Lambda_U &:= \Lambda_0 \backslash \Lambda_+ = \left\{ \sum_{i=1}^{\infty} a_{i} T^{\lambda_i} \in \Lambda_0  \, \bigg{|} \, \frak{v}_T \left( \sum_{i=1}^{\infty} a_{i} T^{\lambda_i} \right) = 0 \right\} \\
\Lambda^* &:= \Lambda \backslash \{0\}. 
\end{align*}

\begin{remark}\label{coventiongradingparameter}
Not using a grading parameter in ~\eqref{Novikovfield}, degrees of structure maps appearing in $A_\infty$-algebras are not well-defined unless considering $\Z / 2$-degree. For this reason, we tacitly use $\Z /2$-grading on the modules appeared in $A_\infty$-algebras in later sections.
\end{remark}

\subsection{$A_\infty$-algebras on toric fibers}\label{Ainftyalgebrasontoricfibers}

Fukaya-Oh-Ohta-Ono \cite{FOOOToric1} constructed an $A_\infty$-algebra on the de-Rham cohomology $H^\bullet(L(\u); \Lambda_0)$ of a toric fiber $L(\u)$ which governs an obstruction of Floer cohomologies on the fiber. The following is a summarization of properties of the constructed $A_\infty$-algebra. 

\begin{theorem} [Section 11 in \cite{FOOOToric1}] \label{Ainftyrelationsum}
Let $X$ be a compact toric symplectic manifold and $L(\u)$ be a toric fiber of $X$. There exists a sequence of maps
\begin{equation}\label{m^k}
\frak{m}^k : H^\bullet(L(\u); \Lambda_0)^{\otimes k} \to H^\bullet(L(\u); \Lambda_0), \quad k \geq 0
\end{equation}
of degree $(2-k)$ mod $2$ satisfying the following properties. 
\begin{enumerate}
\item The pair $\left(H^\bullet (L(\u); \Lambda_0), \{ \frak{m}^k : k \geq 0 \} \right)$ forms a curved $A_\infty$-algebra. Namely, the maps satisfy the $A_\infty$-relation: for every $k \geq 0$ and $h_i \in H^\bullet (L(\u); \Lambda_0)$,
$$
\sum_{k_1, \, k_2} \sum_{i} (-1)^\bigstar \, \frak{m}^{k_2} (h_1, \cdots, h_{i}, \frak{m}^{k_1}(h_{i+1}, \cdots, h_{i+k_1}), h_{i+k_1+1}, \cdots, h_k) = 0
$$ 
where the summation is over all pairs $(k_1, k_2)$ of non-negative integers satisfying $k_1 + k_2 = k + 1$ and $0 \leq i \leq k_2 - 1$, and $\bigstar := | h_1 | + \cdots + | h_i | - i$. Here, $| h_i |$ means the degree of $h_i$.
\item The Poincar\'{e} dual $\textup{PD}[L(\u)]$ is the strict unit in the $A_\infty$-algebra. That is, it obeys
\begin{align*}
&\frak{m}^2 (\textup{PD}[L(\u)], h) =  (-1)^{| h |}\frak{m}^2 (h, \textup{PD}[L(\u)]) = h \quad \text{for all } h\\
&\frak{m}^{k+1} (h_1, \cdots, \textup{PD}[L(\u)], \cdots, h_k) = 0 \quad \text{for any } k \neq 1. 
\end{align*}
\item The $A_\infty$-algebra is canonical (or minimal). It means that $\frak{m}^1_{0} = 0$ \eqref{definitionofmhbeta}.
\item The $A_\infty$-algebra is $G$-gapped for some additive discrete submonoid $G = \{ \lambda_i: \lambda_0 = 0, \{ \lambda_i \}_{i \geq 0} \text{ is increasing, and } \lambda_i \to \infty \}$ of $\R$. That is, $\frak{m}^k$ can be expressed of the form 
$$
\frak{m}^k = \sum_{i=0}^\infty \frak{m}^{k, i} \, T^{\lambda_i}
$$
where $\lambda_i \in G$ and each $\frak{m}^{k, i}: H^\bullet(L(\u); \C)^{\otimes k} \to H^\bullet(L(\u); \C)$ is a $\C$-linear map. 
\end{enumerate}
\end{theorem}

We briefly outline the construction of $\{ \frak{m}^k : k \geq 0 \}$ in \cite{FOOOToric1}. For the torus invariant complex structure $J$ of the symplectic toric manifold $X$ from a moment polytope $P$, we consider the moduli space $\mcal{M}_{k+1}(X, J; L(\u); \beta)$ of stable maps from a bordered genus $0$ Riemann surface to $X$ with the boundary condition $L(\u)$ in the class $\beta \in \pi_2 (X, L(\u))$ together with $k+1$ boundary marked points respecting counter-clockwise orientation. By taking a spin structure on the torus $L(\u)$ and perturbing the moduli space properly, it becomes oriented and transversal.  The perturbed space is still denoted by $\mcal{M}_{k+1}(X, J; L(\u); \beta)$ by an abuse of notation. This space comes with the evaluation maps at marked points:
\begin{align*}
&\ev_+ : \mcal{M}_{k+1}(X, J; L(\u); \beta) \to L(\u)^k, \quad  \ev_+( \varphi, z_0, z_1, \cdots, z_k ) = ( \varphi (z_1), \cdots, \varphi (z_k)) \\ 
&\ev_0 : \mcal{M}_{k+1}(X, J; L(\u); \beta) \to L(\u), \quad \ev_0( \varphi, z_0, z_1, \cdots, z_k ) = \varphi (z_0).
\end{align*}
For complex-valued differential forms $h_1, \cdots, h_k$ on $L(\u)$, we define the map $\frak{m}^k_\beta$ as follows:
\begin{align}\label{definitionofmhbeta}
&\frak{m}^1_{0} (h_1) := (-1)^{n + | h | + 1} dh_1 \quad \mbox{if } (k,\beta) = (1, 0) \\
&\frak{m}^k_\beta (h_1, \cdots, h_k) := (\ev_0)_! \left(\ev_+^* (h_1 \times \cdots \times h_k )\right) \quad \mbox{otherwise}. 
\end{align}
Namely, we first pull the differential form $h_1 \times \cdots \times h_k $ along $\ev_+$ back to $\mcal{M}_{k+1}(X, J; L(\u); \beta)$ and then take the integration along the fiber under $\ev_0$. The outcome turns out to be a differential form on $L(\u)$ because the map $\ev_0$ becomes a submersion by taking a $T^n$-equivariant perturbation such that $\ev_0$ is $T^n$-equivariant. Moreover, due to this feature, the output is also $T^n$-invariant whenever taking $T^n$-invariant differential forms as an input. Since a $T^n$-invariant differential form is a harmornic differential form with respect to a choice of $T^n$-equivariant Riemannian metric on $L(\u)$ and vice versa, by identifying $H^\bullet(L(\u); \C)$ with the set of harmonic differential forms, every map $\frak{m}^k_\beta$ is reduced to the cohomology level. Extending $\frak{m}^k_\beta$ linearly to the completion of $H^\bullet(L(\u); \C) \otimes \Lambda_0$ and declaring
$$
\frak{m}^k := \sum_{\beta} \frak{m}^k_\beta \,\, T^{\omega(\beta)/2\pi},
$$
we obtain the map ~\eqref{m^k}. 

We discuss the deformation of the given $A_\infty$-algebra by a $1$-cochain $b \in H^1 (L(\u); \Lambda_0)$. It is convenient to explain it in two stages. Note that a cochain $b \in H^1 (L(\u); \Lambda_0)$ can be written as $b = b_0 + b_+$ where $b_0 \in H^1 (L(\u); \C)$ and $b_+ \in H^1 (L(\u); \Lambda_+)$. Firstly, adorning $L(\u)$ with a flat complex line bundle $\mcal{L}_{b_0}$ such that the holonomy of a closed curve $\gamma$ is $\exp \left( \int_\gamma b_0 \right)$ as in Cho \cite{Cho}, we deform $\frak{m}^k_\beta$ as follows:
$$
\frak{m}_{\beta; b_0}^k := \exp \left( \int_{\pa \beta} b_0 \right) \cdot \frak{m}_\beta^k.
$$
We then declare
$$
\frak{m}_{b_0}^k := \sum_{\beta} \frak{m}_{\beta; b_0}^k \, T^{\omega(\beta)/ 2\pi}.
$$
Secondly, following \cite{FOOO}, for $b_+ \in H^1 (L(\u); \Lambda_+)$ we define the map $\frak{m}^k_{b_+}$ by
$$
\frak{m}^{k}_{b_+} (h_1, \cdots , h_k) := \sum_{n_0, n_1, \cdots, n_k} \frak{m}^{k+n_0+ \cdots +n_k} \left( b_+^{\otimes n_0}, h_1, b_+^{\otimes n_1}, \cdots, b_+^{\otimes n_{k-1}}, h_k, b_+^{\otimes n_k} \right)
$$
where the summation is taken over all possible integers $n_0, \cdots, n_k \geq 0$. Combining these two deformations, we define the deformed map $\frak{m}^k_b$ to be
$$
\frak{m}^{k}_{b} (h_1, \cdots , h_k) := \sum_{n_0, n_1, \cdots, n_k} \frak{m}_{b_0}^{k+n_0+ \cdots +n_k} \left( b_+^{\otimes n_0}, h_1, b_+^{\otimes n_1}, \cdots, b_+^{\otimes n_{k-1}}, h_k, b_+^{\otimes n_k} \right).
$$
We then obtain the deformed $A_\infty$-algebra $\left(H^\bullet (L(\u); \Lambda_0), \{ \frak{m}^{k}_b : k \geq 0 \} \right)$ on $L(\u)$ with the strict unit $\text{PD}[L(\u)]$. 

One case where $\frak{m}^{1}_b$ becomes a differential is for a cochain $b$ to satisfy
$$
\sum_{k=0}^\infty \frak{m}^k (b^{\otimes k}) \equiv  0 \quad \text{ mod } \Lambda_+ \cdot \text{PD}[L(\u)]
$$
because from the $A_\infty$-relation and the unitality of $\text{PD}[L(\u)]$ it follows that
\begin{align*}
0 &= \frak{m}^{1}_b ( \frak{m}^{1}_b (h) ) + (-1)^{|h| - 1} \frak{m}^{2}_b (h, \frak{m}^{0}_b(1)) + \frak{m}^{2}_b (\frak{m}^{0}_b(1), h) \\
&= \frak{m}^{1}_b ( \frak{m}^{1}_b (h) ) + (-1)^{|h| - 1} \frak{m}^{2}_b \left(h, \sum_{k=0}^\infty \frak{m}^k (b^{\otimes k})\right) + \frak{m}^{2}_b \left(\sum_{k=0}^\infty \frak{m}^k (b^{\otimes k}), h\right) \\
&= \frak{m}^{1}_b ( \frak{m}^{1}_b (h) ) + (-1)^{|h| - 1} (-1)^{|h|}\lambda h + \lambda h \\
&= \frak{m}^{1}_b ( \frak{m}^{1}_b (h) ).  
\end{align*}
for some $\lambda \in \Lambda_+$. Such a cochain $b$ is called a \emph{bounding cochain}.

When $X$ is a compact toric manifold, every $1$-cochain $b$ of $L(\u)$ is turned out to be a bounding cochain (Proposition 4.3 in \cite{FOOOToric1}) and so gives rise to the map $\frak{m}^{1}_b$ satisfying $\frak{m}^{1}_b \circ \frak{m}^{1}_b = 0$. In this case, we can define the \emph{Floer cohomology deformed by a bounding cochain} $b$ as follows:
$$
HF^\bullet ((L(\u), b), (L(\u), b); \Lambda_0) := \frac{\Ker \frak{m}^{1}_b}{\Im \frak{m}^{1}_b}.
$$

\subsection{Potential functions of toric fibers}
We now review the derivation of potential functions. 
\begin{definition} Let $X$ be a compact symplectic toric manifold and $L(\u)$ be a toric fiber of $X$. The \emph{potential function} \ulp or \emph{Landau-Ginzburg superpotential}\urp on $L(\u)$ is the function 
$$
\frak{PO}^\u : H^1 (L(\u); \Lambda_0 ) \to \Lambda_+
$$
determined by the relation
$$
\sum_{k=0}^\infty \frak{m}^k (b^{\otimes k}) = \frak{PO}^\u (b) \cdot \text{PD}[L(\u)].
$$ 
\end{definition}

With the aid of the toric structure, the potential function can be expressed as a Laurent series of variables coming from a basis of the dual lattice $M$ as follows. We fix a basis $\{ e_1, \cdots, e_n \}$ of $M$. Identifying $M$ with $H^1(L(\u); \Z)$ and regarding it as a basis of  $H^1(L(\u); \Z)$, each cochain $b \in H^1(L(\u); \Lambda_0)$ can be written $b = \sum_{i=1}^n x_i e_i$. Thinking $x_i$ as a coordinate function from $H^1(L(\u); \Z)$, we have $x_i = e^*_i$ and set $y_i := \exp (e^*_i)$ where $\{ e_1^*, \cdots, e_n^* \}$ is the dual basis of $\{ e_1, \cdots, e_n \}$. By the dimension reason, notice that $\frak{m}^0_\beta (1)$ can be nonzero only when the Maslov index of $\beta$ is less than or equal to $2$. Since compact symplectic toric manifolds do not admit any non-constant (virtual) holomorphic discs with Maslov index less than or equal to $0$, the (virtual) fundamental chain of $\mcal{M}_{1}(X, J; L(\u); \beta)$ is a cycle and thus $\frak{m}^0_{\beta}(1)$ becomes a multiple of $\text{PD}[L(\u)]$. The multiple is turned out to be independent to the choice of perturbation and in general a rational number. It is called the \emph{one-point open Gromov-Witten invariant} of $\beta$ and denoted by $n_\beta$. Furthermore, by taking a perturbation compatible with the forgetful map of boundary marking points, $\frak{m}^k_\beta (b^{\otimes k})$ is reduced to $\frak{m}^0_\beta (1)$ with weight $\frac{1}{k!}(\pa \beta \cap b)^k$. Thus, we obatin
\begin{align*}
\sum_{k=0}^\infty \frak{m}^k (b^{\otimes k}) &= \sum_{k=0}^\infty \sum_\beta  \frak{m}^k_\beta (b^{\otimes k}) \, T^{\omega(\beta)/ 2 \pi} \\
&= \sum_\beta \sum_{k=0}^\infty \frac{1}{k!} (b \cap \pa \beta)^k \frak{m}_\beta^0(1) T^{\omega(\beta)/ 2\pi} \\
&= \sum_\beta \sum_{k=0}^\infty \frac{1}{k!} (b \cap \pa \beta)^k n_\beta \, T^{\omega(\beta)/ 2\pi} \cdot \text{PD}[L(\u)]
\end{align*}
where the summation of $\beta$ is over all $\beta$ of Maslov index $2$. Since $\exp (b \cap \pa \beta) = \exp (\langle b, \v_j  \rangle) = y_1 ^{v_{j,1}} \cdots y_n ^{v_{j,n}} $ where $\v_j := (v_{j,1}, \cdots, v_{j,n})$, the potential function can be expressed as a Laurent power series in terms of the variable $\{y_1, \cdots, y_n\}$. When regarding the potential function at $\u$ as a Laurent power series, we denote it by $\frak{PO}^\u (\y)$ where $\y = (y_1, \cdots, y_n)$. 

As observed by Cho-Oh in \cite{ChoOh}, the Maslov index of $\beta$ can be counted as the intersection number between a holomorphic disc representing $\beta$ and the toric divisor (See also Auroux \cite{AuT}). Also, it was shown that there exists a unique holomorphic disc of Maslov index 2 without sphere bubbles corresponding to a facet of the polytope and the potential function consequently contains the terms coming from those holomorphic discs. The part of the potential function that can be read off from the facets of the polytope is called the \emph{leading order potential function} (or \emph{Hori-Vafa potential}) and denoted by $\frak{PO}_0^\u$. Letting $\beta_j$ be the homotopy class represented by the holomorphic disc corresponding to a facet $P_j$ where $\pa P = \sum_{j=1}^m P_j$, the open Gromov-Witten invariant $n_{\beta_j}$ is $1$ and the area formula of $\beta_j$ is given by $2 \pi \cdot l_j (\u) = \omega( \beta_j)$ (\cite{FOOOToric1}, \cite{ChoOh}) In terms of the variable $\{ y_1, \cdots, y_n \}$, the leading order potential function is explicitly written as follows:

\begin{definition}
The \emph{leading order potential function} at $\u$ of $P$ is a Laurent polynomial 
\begin{equation}\label{leadingorderpotentialfunctionofP}
\frak{PO}_0^\u (\y) = \sum_{j=1}^m \y^{\v_j} T^{l_j(\u)}
\end{equation}
where $\v_j := (v_{j,1}, \cdots, v_{j,n})$ is the primitive inward vector normal to $P_j$ and $\y^{\v_j} := y_1^{v_{j,1}} \cdots y_n^{v_{j,n}}$.
\end{definition}

For our usage in Section~\ref{Tropicalizations}, we mention the difference between the potential function and the leading order potential function.  

\begin{theorem} [Theorem 8.2 in \cite{FOOOSurv}] \label{differenceofpotentialfunction}
The difference $\frak{PO}^\u - \frak{PO}^\u_0$ between the potential function and the leading order potential function can be expressed as
$$
\frak{PO}^\u (\y) - \frak{PO}^\u_0 (\y) = \sum_{k} \frak{P}_k (\z) \, T^{ \rho_k}
$$
satisfying the followings: the sequence $\{ \rho_k : k \geq 1 \}$ is a monotonically (not strictly) increasing sequence of positive numbers and goes to infinity if the difference consists of infinitely many terms, and $\frak{P}_k (\z)$ is a non-constant monomial with $\Q$-coefficient in terms of $\z = (z_1, \cdots, z_m)$ where $z_j := \y^{\v_j}T^{l_j(\u)}$.

\end{theorem}

\subsection{Non-displacement of toric fibers}

Our next goal is detecting which toric fiber has a non-trival Floer cohomology. It is related to a critical point of the potential function.\footnote{Theorem~\ref{Ainftyrelationsum}(3) is used for this relation.}

\begin{theorem} [Theorem 4.10 in \cite{FOOOToric1}] \label{criticalandnontrivial}
Let $\u$ be a point in the interior of a moment polytope $P$. The followings are equivalent. 
\begin{enumerate}
\item The potential function $\frak{PO}^\u$ has a critical point on $(\Lambda_U)^n$. That is, the system of equations $y_i \frac{\pa (\frak{PO}^\u)}{\pa y_i} = 0$ admits a solution on $(\Lambda_U)^n$.
\item There exists a cochain $b \in H^1(L(\u); \Lambda_0)$ such that the deformed Floer cohomology $HF((L(\u), b), (L(\u), b); \Lambda_0)$ is isomorphic to $H(T^n; \Lambda_0)$.
\end{enumerate}
\end{theorem}

\begin{corollary} [Proposition 4.12 in \cite{FOOOToric1}] \label{stronglybalancednondisplaceable}
If the potential function at $\u$ of $P$ has a critical point on $(\Lambda_U)^n$, then the toric fiber $L(\u)$ is non-displaceable. 
\end{corollary}

\begin{definition} \label{stronglybalanced}
A toric fiber $L(\u)$ is called \emph{strongly balanced} if the potential function at $\u$ of $P$ has a critical point on $(\Lambda_U)^n$.\footnote{The notion of \emph{balanced} fibers was firstly introduced in \cite{ChoOh}}
\end{definition}

To include more non-displaceable toric fibers, Fukaya-Oh-Ohta-Ono introduced a notion of $\emph{balanced}$ fibers in \cite{FOOOToric1}. To present the definition, we need two auxiliary definitions. 

\begin{definition} Let $P$ be the polytope in $M_\R$ given by 
$$
\displaystyle P = \bigcap_{j=1}^m \left\{ \u \in M_\R : \langle \u, \v_j \rangle - \lambda_j \geq 0 \right\}.
$$ 
For a face $Q$ of $P$, we set $\sigma_Q$ to be the cone generated by $\{ \v_j : \text{a facet } P_j \text{ contains } Q \}$ in $N_\R$. The \emph{normal fan} $\Sigma_P$ to $P$ is the collection of all cones $\sigma_Q$ where $Q$ is a face of $P$. 
\end{definition}

For properties and examples of normal fans, one is referred to $\S 2.3.$ in \cite{CLS}. 

\begin{definition}\label{Hausdorffdistance}
For a pair of polyhedra $(Q_1, Q_2)$ in $\R^n$, the \emph{Hausdorff distance} $d_{\text{Haus}}(Q_1, Q_2)$ between $Q_1$ and $Q_2$ is defined to be
$$
d_{\text{Haus}}(Q_1, Q_2) = \max \left\{ \sup_{q_1 \in Q_1} \left( \inf_{q_2 \in Q_2} d(q_1, q_2) \right), \sup_{q_2 \in Q_2} \left( \inf_{q_1 \in Q_1} d(q_1, q_2) \right) \right\}.
$$
Here, $d$ in the right hand side denotes the Euclidean distance in $\R^n$. 
\end{definition}

Here, a \emph{polyhedron} is meant to be the intersection of finitely many closed half-spaces. Note that a polyhedron is not necessarily compact. 

\begin{definition}[Definition 4.11 in \cite{FOOOToric1}]\label{balanced}
Let $\Sigma \subset N_\R$ be the normal fan of a moment polytope $P$. Let $X_\Sigma$ be the compact complex toric manifold given by the fan $\Sigma$. A toric fiber $L(\u)$ is called $\emph{balanced}$ if there exists a sequence of triples $(\omega^{(i)}, P^{(i)}, \u^{(i)})$ such that 
\begin{enumerate}
\item Each symplectic form $\omega^{(i)}$ is a torus-invariant symplectic form associated to $X_\Sigma$ and the sequence of symplectic forms $\omega^{(i)}$ converges to $\omega$.
\item The normal fan of each $P^{(i)}$ coincides with the normal fan $\Sigma$ of $P$ and the sequence of polytopes $P^{(i)}$ converges to $P$ with respect to the Hausdorff distance.
\item The sequence of positions $\u^{(i)} \in \text{Int}(P^{(i)})$ converges to $\u \in \text{Int}(P)$. 
\item For every sufficiently large integer $N$, there exist an integer $\iota$ and a $1$-cochain $b^{(i)}_N \in H^1(L(\u^{(i)}); \Lambda_0)$ (depending on both $i$ and $N$) such that 
$$
HF^\bullet \left( (L(\u^{(i)}), b^{(i)}_N), (L(\u^{(i)}), b^{(i)}_N); \Lambda_0 / T^N \right) \simeq H^\bullet \left( T^n; \Lambda_0 / T^N \right)
$$
whenever $i \geq \iota$. 
\end{enumerate}
\end{definition}

\begin{theorem} [ Proposition 4.12 in \cite{FOOOToric1}] \label{balancedisnondisplaceable}
If $L(\u)$ is balanced, then $L(\u)$ is non-displaceable.
\end{theorem}

\section{Tropicalizations}\label{Tropicalizations}
The goal of this section is to recall some notions and results from tropical geometry which can be found in \cite{OP}, \cite{EKL}, \cite{MS} for example. The results will be applied to our situation to see which Lagrangian toric fibers are non-displaceable. Moreover, we define tropicalizations relative to a lattice point, which is relevant to our application detecting non-displaceability of toric fibers in later sections. 

Even if the notions and the results in tropical geometry is stated over a more general field in \cite{OP}, we are only working over the Novikov field $\Lambda$ over $\C$ because it is enough for our purpose. Throughout this section, we fix a basis for $M$ so that $M_\R$ (resp. $N_\R$) will be identified with $\R^n$ (resp. $\R^n$) without any mention of the identification.

\subsection{Lifting of tropicalizations to intersection points}\label{Lifting of tropicalizations to intersection points}
We begin by recalling some definitions in \cite{EKL}, \cite{MS}. Suppose that we are given a nonzero Laurent polynomial over $\Lambda$
$$
\frak{P}(\x) = \sum_{\v \in N} a_\v \x^\v \in \Lambda [x_1^\pm, \cdots, x_n^\pm]
$$
where $\x := (x_1, \cdots , x_n), \v := (v_1, \cdots, v_n)$ and $\x^\v = x_1^{v_1} \cdots x_n^{v_n}$. We consider the piecewise-linear function $\trop^\frak{P}$ given by
$$\trop^\frak{P} : M_\R \to \R, \quad 
\displaystyle \u \mapsto \min_{ \v \in N } \left\{ \frak{v}_T ( a_\v ) + \langle \u , \v \rangle  \right\}.
$$
\begin{definition}\label{tropicalizationofequation}
The \emph{tropicalization} of $\frak{P}$ is defined to be the non-differentiable locus of $\trop^\frak{P}$ and is denoted by $\Trop (\frak{P})$.
\end{definition}

\begin{definition} Let $\frak{v}_T^n : (\Lambda^*)^n \to M_\R$ be the component-wise valuation map given by
$$
\frak{v}_T^n (x_1, \cdots, x_n) := (\frak{v}_T (x_1), \cdots , \frak{v}_T (x_n)).
$$
For a variety $X := \text{Spec} \left( \Lambda [ x_1^\pm, \cdots , x_n^\pm ] ) / I \right) $ in the torus $(\Lambda^*)^n$ where $I$ is an ideal of $\Lambda [ x_1^\pm, \cdots , x_n^\pm ]$, the \emph{tropicalization} of $X$ is defined to be the closure of the image of $X$ under $\frak{v}_T^n$ in $M_\R$ and is denoted by $\T(X)$.
\end{definition}

Einsiedler-Kapranov-Lind in \cite{EKL} (See also Theorem 3.1.3. in \cite{MS}) proved that if $X$ is the hypersurface given by $\frak{P} \in \Lambda [x_1^\pm, \cdots, x_n^\pm]$, then $\Trop(\frak{P})$ is equal to $\T(X)$ at least as a set. In this case, we have explicit description of the tropicalization $\T(X) = \Trop(\frak{P})$ as follows.

\begin{proposition} [Theorem 3.1.3. in \cite{MS}] \label{explicitdescriptionoftropicalizationoffunction}The tropicalization $\Trop(\frak{P})$ is the collection of points $\u \in M_\R$ on which the minimum in $\trop^\frak{P}$ is attained by at least two different terms $a_{\v_1} \x^{\v_1}, a_{\v_2} \x^{\v_2}$ with $a_{\v_1}, a_{\v_2} \neq 0$ so that 
$$
\min_{ \v \in N } \left\{ \frak{v}_T ( a_\v ) + \langle \u , \v \rangle  \right\} = \frak{v}_T ( a_{\v_1} ) + \langle \u , \v_1 \rangle =  \frak{v}_T ( a_{\v_2} ) + \langle \u , \v_2 \rangle.
$$
\end{proposition}

In some cases, the intersection of tropicalizations of varieties in $(\Lambda^*)^n$ lifts to intersection points of varieties. According to the main result of Osserman and Payne in \cite{OP}, it happens when $\T(X)$ and $\T(X^\prime)$ \emph{intersect properly} at $\u \in N_\R$.

\begin{definition}\label{intersectproperly}
Tropicalizations $\T(X)$ and $\T(X^\prime)$ \emph{intersect properly} at $\u$ if $\T(X) \cap \T(X^\prime)$ has exactly codimension $\textup{codim} \, X + \textup{codim} \, X^\prime$ in a neighborhood of $\u$. 
\end{definition}

\begin{theorem} [Theorem 1.1. in \cite{OP}] \label{Tropicallifting} If tropicalizations $\T(X)$ and $\T(X^\prime)$ intersect properly at $\u$, then $\u$ is contained in $\T(X \cap X^\prime)$. 
\end{theorem}

For our application, we need to deal with the intersection of more than two tropicalizations and so we need to use the following more specific result in \cite{OP} which fits into our situation well. 

\begin{theorem} [Theorem 5.2.3 in \cite{OP}] \label{Tropicallifting2}
Let $X_1, \cdots, X_n$ be hypersurfaces in the torus $(\Lambda^*)^n$. Suppose that $\u$ is an isolated point of the intersection of tropicalizations $\T(X_1), \dots, \T(X_n)$. Then $\u$ is contained in $\T(X_1 \cap \cdots \cap X_n)$.
\end{theorem}

\subsection{Lifting of tropicalizations to balanced fibers}

Let $X$ denote the compact toric manifold determined by a moment polytope $P$.  As in ~\eqref{descriptionofpolytope}, the moment polytope $P$ has the unique description 
$$
P = \bigcap_{j=1}^m \left\{ \u \in M_\R : l_j (\u) \geq 0 \right\}
$$
where $l_j (\u) := \langle \u, \v_j \rangle - \lambda_j$. Based on Corollary~\ref{stronglybalancednondisplaceable}, we try to find positions $\u \in \text{Int}( P)$ on which the following system of equations admits a solution in $(\Lambda_U)^n$:
$$
y_i \frac{\pa \, \frak{PO}^\u}{\pa y_i} = 0, \quad i = 1, \cdots, n.
$$
Since the potential function $\frak{PO}^\u$ is in general not a Laurent polynomial, we cannot directly apply the results in Section~\ref{Lifting of tropicalizations to intersection points}. We instead consider the leading order potential function $\frak{PO}^\u_0$, which is always a Laurent polynomial, and we deal with the following system of equations 
\begin{equation}\label{logderivativeleadingorderpotentialfun}
y_i \frac{\pa \, \frak{PO}_0^\u}{\pa y_i} = 0, \quad i = 1, \cdots, n.
\end{equation}

Furthermore, to fit this story into Section~\ref{Lifting of tropicalizations to intersection points}, setting 
\begin{equation}\label{relationbetweenxandy}
x_i := y_i \cdot T^{u_i}
\end{equation}
and plugging $y_i := x_i \cdot T^{- u_i}$ into $\frak{PO}_0^\u$, we obtain the expression
\begin{equation}\label{potentialxvariable}
\frak{PO}_0^\u = \sum_{\v \in N} a_\v \x^\v.
\end{equation}
Note that each coefficient $a_\v \in \Lambda$ of $\x_\v$ does not depend on the position $\u$. To emphasize the non-dependence on positions of the coefficients in $\frak{PO}_0^\u$, the superscript $\u$ in the potential function will be omitted when regarding the potential function as a Laurent polynomial with respect to $\x$:
$$
\frak{PO}_0(\x) := \sum_{\v \in N} a_\v \x^\v.
$$
We then have the following relation between critical points with resepect to $\x$ and critical points with resepect to $\y$. 

\begin{lemma}\label{yixicritialpoints} The system of equations in ~\eqref{logderivativeleadingorderpotentialfun} has a solution $\{(y_i) : y_i \in \Lambda_U\}$ for $\u \in \textup{Int}(P)$ if any only if the system of equations
$$
x_i \frac{\pa \frak{PO}_0}{\pa x_i} = 0, \quad i = 1, \cdots, n.
$$
has a solution $\{(x_i) : x_i \in \Lambda^* \}$ for $\frak{v}_T^n (\x) = \u \in \textup{Int}(P)$. 
\end{lemma}

The following lemma roughly says that inside of the polytope $P$ the leading order potential function completely determines the tropicalization of potential function.

\begin{lemma}\label{tropicalizationofpotentialleading}
Let $P$ be a polytope. For any sufficiently large integer $N$, two tropicalizations $\Trop \left(x_i \frac{\pa \frak{PO}_0}{\pa x_i}\right) \cap P$ and $\Trop \left( x_i \frac{\pa \frak{PO}}{\pa x_i} \textup{ mod } T^N \right) \cap P$ coincide as a set.
\end{lemma}
\begin{proof}
We need to show that any terms in the difference $\frak{PO} - \frak{PO}_0$ do not contribute to the tropicalizations. By Theorem~\ref{differenceofpotentialfunction}, one can regard $\frak{PO} (\textup{mod } T^N)  - \frak{PO}_0$ as a polynomial with respect to $\{ z_j := \y^{\v_j} T^{l_j(\u)} = \x^{\v_j} T^{-\lambda_j} : j = 1, \cdots, m \}$. Whenever $\frak{v}_T^n (\x) \in \textup{Int}(P)$ (implying $\frak{v}_T (z_j) > 0$ for any $j$), any monomial of the form $T_{}^{\rho_k} z_{j_1}^{k_1} \cdots z_{j_r}^{k_r}$ with $r \in \N, k_j \in \N$ and $\rho_k > 0$ satisfies
\begin{equation}\label{difference1}
\frak{v}_T \left( T_{}^{\rho_k} z_{j_1}^{k_1} \cdots z_{j_r}^{k_r} \right) >  \frak{v}_T \left( z_{j_i} \right)
\end{equation}
for each $i$ with $1 \leq i \leq r$. Since $P$ is compact, we can take a sufficiently large integer $N$ so that $N$ is greater than $\sup \left( \max_{j} \frak{v}_T (z_j) \right)$ where the supremum is taken over all $(z_1, \cdots, z_m)$ satisfying $\frak{v}_T^n (\x) \in P$. Such a choice of $N$ guarantees the presence of all $z_j$'s in $\frak{PO}_0 \textup{ mod } T^N$ so that any terms in the difference do not involve in $\Trop \left( x_i \frac{\pa \frak{PO}}{\pa x_i} \textup{ mod } T^N \right)$ because of ~\eqref{difference1}.
\end{proof}

We then obtain a corollary of Thoerem~\ref{Tropicallifting2}.
\begin{corollary}\label{Tropicaliftingbalanced}
If $\u \in \textup{Int}(P)$ is an isolated point of the intersection of tropicalizations $\Trop \left(x_1 \frac{\pa \frak{PO}_0}{\pa x_1 }\right),\cdots, \Trop \left(x_n \frac{\pa \frak{PO}_0}{\pa x_n }\right)$, then the fiber $L(\u)$ is balanced (See Definition~\ref{balanced}) and hence non-displaceable.

If in addition the potential function $\frak{PO}$ is a Laurent \emph{polynomial}, then the fiber $L(\u)$ is strongly balanced \ulp See Definition~\ref{stronglybalanced}\urp and hence non-displaceable.  
\end{corollary}

\begin{proof}
We take the constant sequence of triples $(\omega^{(i)}, P^{(i)}, \u^{(i)})$ by setting $\omega^{(i)} := \omega$, $P^{(i)} := P$ and $\u^{(i)} := \u$ for all $i$. For sufficiently large interger $N$, two tropicalizations $\Trop \left(x_i \frac{\pa \frak{PO}_0}{\pa x_i}\right) \cap P$ and $\Trop \left( x_i \frac{\pa \frak{PO}}{\pa x_i} \textup{ mod } T^N \right) \cap P$ coincide by Lemma~\ref{tropicalizationofpotentialleading}. Then, the point $\u$ is still an isolated point of the intersection of $\Trop \left(x_1 \frac{\pa \frak{PO}_0}{\pa x_1 } \textup{ mod } T^N \right),\cdots, \Trop \left(x_n \frac{\pa \frak{PO}_0}{\pa x_n } \textup{ mod } T^N \right)$. By Theorem~\ref{Tropicallifting2} and Lemma~\ref{yixicritialpoints}, there is a critical point of $\frak{PO} \textup{ mod } T^N$ with respect to $\y$. By the modulo $N$ version of Theorem~\ref{criticalandnontrivial}, we conclude that 
$$
HF((L(\u^{(i)}), b^{(i)}_N), (L(\u^{(i)}), b^{(i)}_N); \Lambda_0 / T^N ) \simeq H(T^n; \Lambda_0 / T^N).
$$
Therefore, the fiber $L(\u)$ is balanced. Finally, non-displacement of $L(\u)$ follows from Theorem~\ref{balancedisnondisplaceable}. 

If the potential function $\frak{PO}$ is a Laurent polynomial, we can directly apply Theorem~\ref{Tropicallifting2} without modulo $N$ argument. 

\end{proof}

Note that the cases satisfying $\frak{PO}$ is a Laurent polynomial in Corollary~\ref{Tropicaliftingbalanced} include more than the Fano cases in which the potential function $\frak{PO}$ is exactly same as the leading order potential function $\frak{PO}_0$. According to results in Chan-Lau \cite{CL} and Auroux \cite{Au}, all semi-Fano surfaces and Hirzebruch surface of degree $3$ are examples satisfying the additional assumption. 

\subsection{Tropicalizations of $P$ relative to $\m$}
Let $\frak{PO}_0^\u$ be the leading order potential function at $\u$ of $P$. Putting $\frak{PO}_0^\u (\x) := \sum_{\v \in N} a_\v \x^\v$, for a lattice point $\m \in M$ we consider the piecewise-linear function 
\begin{align*}
\trop^{P , \, \frak{m}} : M_\R \to \R, \quad \u \mapsto &\min \left\{ \frak{v}_T (a_\v) + \langle \u , \v \rangle :  \v \in N \text{ such that } \langle \m, \v \rangle \neq 0  \right\} \\
&= \min \left\{ l_j (\u) : \langle \m, \v_j \rangle \neq 0 \right\}.
\end{align*}

\begin{definition}\label{tropicalizationexcepting}
The \emph{tropicalization} of $P$ \emph{relative to} $\m$ is defined to be the non-differentiable locus of $\trop^{P, \, \frak{m}}$ and is denoted by $\Trop \left(P, \, \m \right)$. For simplicity, we put $\Trop \left(\m \right) := \Trop \left(P, \, \m \right)$ unless the omission causes any confusion. 
\end{definition}

The tropicalization of $P$ relative to $\m$ can be understood as the tropicalization of the logarithmic derivative of the leading order potential function $\frak{PO}_0$ with respect to the direction $\m$. It naturally generalizes  $\Trop \left( {x_i \frac{\pa \frak{PO}_0}{\pa x_i}} \right)$ because $\Trop \left( {x_i \frac{\pa \frak{PO}_0}{\pa x_i}} \right) = \Trop (P, \, \textbf{e}_i)$ where $\textbf{e}_i$ is a member of the standard unit vectors. 

It is worthwhile to mention explicit description of $\Trop \left(P , \, \m \right)$ which will be frequently used in Section~\ref{ProofA} and Section~\ref{ProofB}.
\begin{proposition}\label{descriptionoftropcurve}
The tropicalization $\Trop \left(P, \, \m \right)$ is the collection of points $\u \in M_\R$ on which the minimum of $\trop^{P , \, \frak{\m}}$ is attained by at least two $l_{j_1}, l_{j_2}$  with $j_1 \neq j_2$ so that  $\langle \m, \v_{j_1} \rangle \neq 0, \langle \m, \v_{j_2} \rangle \neq 0$ and
$$
\min_j \left\{ l_j (\u) :  \langle \m, \v_j \rangle \neq 0  \right\} = l_{j_1}(\u) = l_{j_2}(\u).
$$
\end{proposition}

For later purpose, we keep one more notation. Let $Q$ be a polyhedron, which is the intersection of finitely many closed half-spaces.  At least formally, using ~\eqref{leadingorderpotentialfunctionofP} and ~\eqref{relationbetweenxandy}, the leading order potential function $\frak{PO}_0$ of $Q$ can be written down from the defining equations of the plane containing the facets of $Q$. 

\begin{definition}
The \emph{tropicalization} of $Q$ is defined to be the tropicalization of $\frak{PO}_0$ (See Definition~\ref{tropicalizationofequation}) and is denoted by $\Trop(Q)$. 
\end{definition}

Note that the tropicalization $\Trop(P, \m)$ relative to $\m$ is the tropicalization of the polyhedron $Q$ determined by $\left\{ l_j (\u) :  \langle \m, \v_j \rangle \neq 0 \right\}$.

\section{Examples}\label{examples}

\begin{example}\label{Onepointblowup1}
Let $X$ be the one-point blowup of $\C\P^2$ determined by the moment polytope
$$
P = \left\{ (u_1, u_2) \in M_\R :   u_1 \geq 0,\, u_2 \geq 0,\, 1 - u_1 - u_2 \geq 0,\, c - u_2 \geq 0 \right\}.
$$
where $c$ is a real number with $0 < c < 1$. In \cite{FOOOToric1}, Fukaya-Oh-Ohta-Ono showed how the positions of toric fibers having a non-trivial Floer cohomology change as $c$ varies (Figure~\ref{NondisfiberoneptblowupCP2}). 

\begin{figure}[h]
\centering
\begin{tikzpicture}
 \def\size{0.75}
 \def\trans{5}

 \coordinate (A) at (0, 0);
 \coordinate (B) at (4 * \size, 0);
 \coordinate (C) at (4/3 * \size, 8/3 *\size);
 \coordinate (D) at (0, 8/3 * \size);

 \draw [very thick] (A) -- (B) -- (C) -- (D)-- (A);
 \node [black] at (4/3 * \size, 4/3 * \size) {\tiny$\bullet$}; 

 \coordinate (A) at (0- \trans, 0);
 \coordinate (B) at (4 * \size- \trans, 0);
 \coordinate (C) at (1 * \size- \trans, 3 *\size);
 \coordinate (D) at (0- \trans, 3 * \size);

 \draw [very thick] (A) -- (B) -- (C) -- (D)-- (A);
 \node [black] at (4/3 * \size- \trans, 4/3 * \size) {\tiny$\bullet$}; 
 \node [black] at (\size- \trans, 2 * \size) {\tiny$\bullet$}; 

 \coordinate (A) at (0+ \trans, 0);
 \coordinate (B) at (4 * \size+ \trans, 0);
 \coordinate (C) at (2 * \size+ \trans, 2 *\size);
 \coordinate (D) at (0 + \trans, 2 * \size);

 \draw [very thick] (A) -- (B) -- (C) -- (D)-- (A);
 \node [black] at (3/2 * \size+ \trans, \size) {\tiny$\bullet$}; 
\end{tikzpicture}
\caption{Non-displaceable toric fibers: $c< \frac{1}{3}, \, c = \frac{1}{3}, \, c > \frac{1}{3}$ }\label{NondisfiberoneptblowupCP2}
\end{figure}
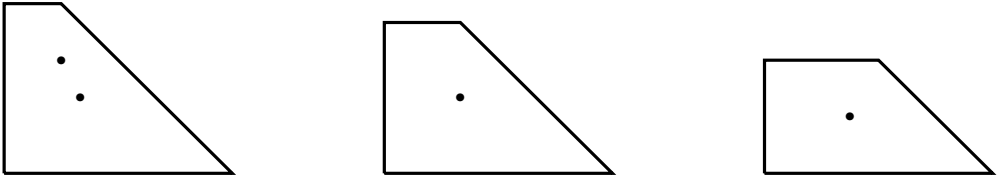 

In this case, the following tropicalizations intersect properly regardless of the value $c$ as seen in Figure~\ref{OnepointblowupFig}:
$$
\Trop \left( x_1 \frac{\pa \frak{PO}_0}{\pa x_1} \right) \, \cap \, \Trop \left( x_2 \frac{\pa \frak{PO}_0}{\pa x_2} \right) \cap \textup{ Int} (P)
$$ 
and thus Corollary~\ref{Tropicaliftingbalanced} can be applied. Therefore, the toric fibers over the intersection are non-displaceable and the intersection exactly locates strongly balanced fibers since $X$ is Fano. Moreover, the tropicalizations illustrates how the number of non-displaceable toric fibers changes as a $c$ changes.

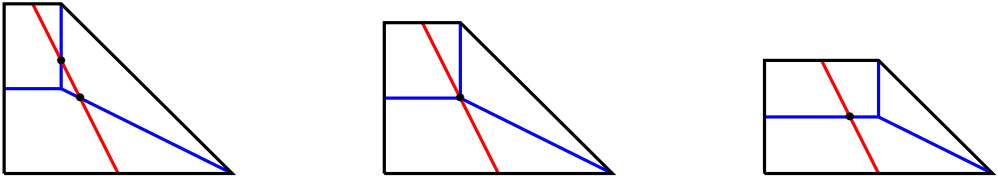
\begin{figure}[h]
\centering
\begin{tikzpicture}
 \def\size{0.75}
 \def\trans{5}

 \coordinate (A) at (0, 0);
 \coordinate (B) at (4 * \size, 0);
 \coordinate (C) at (4/3 * \size, 8/3 *\size);
 \coordinate (D) at (0, 8/3 * \size);
 \coordinate (E) at (4/3 * \size , 4/3 * \size);
 \coordinate (G) at (8/3 * \size, 4/3 * \size);
 \coordinate (H) at (0, 4/3 * \size);
 \coordinate (I) at (2/3 * \size, 8/3 * \size);
 \coordinate (J) at (2 * \size, 0);  

 \draw [very thick, color=blue] (E) -- (H) ;
 \draw [very thick, color=blue] (E) -- (C) ;
 \draw [very thick, color=blue] (E) -- (B) ;
 \draw [very thick, color=red] (I) -- (J) ;
 \node [black] at (4/3 * \size, 4/3 * \size) {\tiny$\bullet$}; 
 \draw [very thick] (A) -- (B) -- (C) -- (D)-- (A); 

 \coordinate (A) at (0- \trans, 0);
 \coordinate (B) at (4 * \size- \trans, 0);
 \coordinate (C) at (1 * \size- \trans, 3 *\size);
 \coordinate (D) at (0- \trans, 3 * \size); 
 \coordinate (E) at (3/2 * \size- \trans , 3/2 * \size);
 \coordinate (E1) at (1 * \size- \trans , 3/2 * \size);
 \coordinate (G) at (5/2 * \size- \trans, 3/2 * \size);
 \coordinate (H) at (0- \trans, 3/2 * \size);
 \coordinate (I) at (1/2 * \size- \trans, 3 * \size);
 \coordinate (J) at (2 * \size- \trans, 0);  
 
 \draw [very thick, color=blue] (E1) -- (H) ;
 \draw [very thick, color=blue] (E1) -- (C) ;
 \draw [very thick, color=blue] (E1) -- (B) ;
 \draw [very thick, color=red] (I) -- (J) ;
 \node [black] at (4/3 * \size- \trans, 4/3 * \size) {\tiny$\bullet$}; 
 \node [black] at (\size- \trans, 2 * \size) {\tiny$\bullet$}; 
\draw [very thick] (A) -- (B) -- (C) -- (D)-- (A); 

 \coordinate (A) at (0+ \trans, 0);
 \coordinate (B) at (4 * \size+ \trans, 0);
 \coordinate (C) at (2 * \size+ \trans, 2 *\size);
 \coordinate (D) at (0 + \trans, 2 * \size);
 \coordinate (E) at (\size+ \trans , \size);
 \coordinate (E1) at (2 * \size+ \trans , \size);
 \coordinate (G) at (3* \size+  \trans, \size);
 \coordinate (H) at (0 + \trans, \size);
 \coordinate (I) at ( \size + \trans, 2 * \size);
 \coordinate (J) at (2 * \size + \trans, 0);  
 
 \draw [very thick, color=blue] (E1) -- (H) ;
 \draw [very thick, color=blue] (E1) -- (C) ;
 \draw [very thick, color=blue] (E1) -- (B) ; 
 \draw [very thick, color=red] (I) -- (J) ;
 \node [black] at (3/2 * \size+ \trans, \size) {\tiny$\bullet$}; 
 \draw [very thick] (A) -- (B) -- (C) -- (D)-- (A); 

\end{tikzpicture}
\caption{The intersection of tropicalizations: $c< \frac{1}{3}, \, c = \frac{1}{3}, \, c > \frac{1}{3}$ }\label{OnepointblowupFig}
\end{figure}

\end{example}

\begin{example}\label{Twopointblowup0}
Let $X$ be the two-point blowup of $\C\P^2$ determined by the moment polytope
$$
P = \left\{ (u_1, u_2) \in M_\R :   u_1 \geq 0,\, -1/4 + u_1 + u_2 \geq 0,\, u_2 \geq 0,\, 1 - u_1 - u_2 \geq 0,\, 1/2 - u_2 \geq 0 \right\}.
$$
This is one of symplectic toric manifolds admitting a continuum family of non-displaceable toric fibers presented in \cite{FOOOToric2}. Specifically, Fukaya-Oh-Ohta-Ono proved that the fibers over $\{ (u_1, u_2) \in M_\R : \frac{1}{4} \leq u_1 \leq \frac{3}{8}, u_2 = \frac{1}{4}\}$ are non-displaceable as in Figure~\ref{NondisfibertwoptblowupCP2} (See also Wilson-Woodward \cite{WW} and Abreu-Macarini \cite{AM}). 

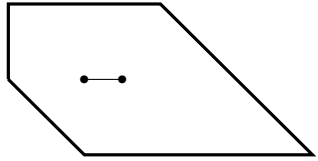
\begin{figure}[h]
\centering
\begin{tikzpicture}
 \def\size{1}
 \def\trans{5}

 \coordinate (A) at (0, \size);
 \coordinate (B) at (\size, 0);
 \coordinate (C) at (4 * \size, 0);
 \coordinate (D) at (2 * \size, 2 *\size);
 \coordinate (E) at (0, 2 * \size);
 \coordinate (F) at (1.5 * \size , \size);
 \coordinate (G) at (\size, \size);
 \coordinate (H) at (2 * \size, \size);
 \coordinate (I) at (0, 1.5 * \size);
 \coordinate (J) at (\size, 2 * \size);  
 \coordinate (K) at (0, \size);
 \coordinate (L) at (4 * \size, 0);
 \coordinate (M) at (3 * \size, \size);
 \coordinate (N) at (1/2 * \size, 1/2 * \size);

 \draw [thin, color=black] (G) -- (F) ;
 \node [black] at (\size, \size) {\tiny$\bullet$}; 
 \node [black] at (1.5* \size, \size) {\tiny$\bullet$}; 
 \draw [very thick] (A) -- (B) -- (C) -- (D)-- (E)-- (A);

\end{tikzpicture}
\caption{Non-displaceable toric fibers}\label{NondisfibertwoptblowupCP2}
\end{figure}

In this case, we cannot use Corollary~\ref{Tropicaliftingbalanced} because the tropicalizations do not intersect properly as in Figure~\ref{TwopointblowupFig0}, but yet we can apply Theorem~\ref{THEOREMA}. The primary subspaces for $P$ are $\langle (0,1)\rangle, \langle (1,0)\rangle$ and $\langle (1,1)\rangle$. By finding the intersection of the following tropicalizations: 
\begin{align*}
\Trop \left((0,1) \right) \cap \Trop \left((1,0) \right) \cap \Trop \left((1,1) \right) \cap \textup{ Int} (P),
\end{align*} 
we can detect non-displaceable toric fibers. Three tropicalizations are drawn in Figure~\ref{TwopointblowupFig} and the intersection of the tropicalizations is exactly $\{ (u_1, u_2) \in M_\R : \frac{1}{4} \leq u_1 \leq \frac{3}{8}, u_2 = \frac{1}{4}\}$ as desired.
\begin{figure}[h]
\centering
\begin{tikzpicture}
 \def\size{1}
 \def\trans{5}

 \coordinate (A) at (0, \size);
 \coordinate (B) at (\size, 0);
 \coordinate (C) at (4 * \size, 0);
 \coordinate (D) at (2 * \size, 2 *\size);
 \coordinate (E) at (0, 2 * \size);
 \coordinate (F) at (1.5 * \size , \size);
 \coordinate (G) at (\size, \size);
 \coordinate (H) at (2 * \size, \size);
 \coordinate (I) at (0, 1.5 * \size);
 \coordinate (J) at (\size, 2 * \size);  
 \coordinate (K) at (0, \size);
 \coordinate (L) at (4 * \size, 0);
 \coordinate (M) at (3 * \size, \size);
 \coordinate (N) at (1/2 * \size, 1/2 * \size);

 \draw [very thick, color=red] (G) -- (H) ;
 \draw [very thick, color=red] (H) -- (C) ;
 \draw [very thick, color=red] (H) -- (D) ;
 \draw [very thick, color=red] (G) -- (B) ;
 \draw [very thick, color=red] (G) -- (I) ;
 \draw [thin, color=black] (G) -- (F) ;
 \node [black] at (\size, \size) {\tiny$\bullet$}; 
 \node [black] at (1.5* \size, \size) {\tiny$\bullet$}; 
 \draw [very thick] (A) -- (B) -- (C) -- (D)-- (E);
 \draw [very thick, dashed] (E)-- (A);

 \coordinate (A) at (0 - \trans, \size);
 \coordinate (B) at (\size - \trans, 0);
 \coordinate (C) at (4 * \size - \trans, 0);
 \coordinate (D) at (2 * \size - \trans, 2 *\size);
 \coordinate (E) at (0 - \trans, 2 * \size);
 \coordinate (F) at (1.5 * \size - \trans , \size);
 \coordinate (G) at (\size - \trans, \size);
 \coordinate (H) at (2 * \size - \trans, \size);
 \coordinate (I) at (0 - \trans, 2 * \size);
 \coordinate (J) at (\size - \trans, 2 * \size);  
 \coordinate (K) at (0 - \trans, \size);
 \coordinate (L) at (2.5 * \size - \trans, 0);
 \coordinate (M) at (3 * \size - \trans, \size);
 \coordinate (N) at (1/2 * \size - \trans, 1/2 * \size);
 
 \draw [very thick, color=blue] (F) -- (J) ;
 \draw [very thick, color=blue] (F) -- (K) ;
 \draw [very thick, color=blue] (F) -- (L) ;
 \draw [thin, color=black] (G) -- (F) ;  
 \node [black] at (\size- \trans, \size) {\tiny$\bullet$}; 
 \node [black] at (1.5* \size- \trans, \size) {\tiny$\bullet$}; 
 \draw [very thick] (A) -- (B);
 \draw [very thick] (C) -- (D);
 \draw [very thick] (E)-- (A);
 \draw [very thick, dashed] (B) -- (C);
 \draw [very thick, dashed] (D) -- (E);

 \coordinate (A) at (0 + \trans, \size);
 \coordinate (B) at (\size + \trans, 0);
 \coordinate (C) at (4 * \size + \trans, 0);
 \coordinate (D) at (2 * \size + \trans, 2 *\size);
 \coordinate (E) at (0 + \trans, 2 * \size);
 \coordinate (F) at (1.5 * \size + \trans , \size);
 \coordinate (G) at (\size + \trans, \size);
 \coordinate (H) at (2 * \size + \trans, \size);
 \coordinate (I) at (0 + \trans, 2 * \size);
 \coordinate (J) at (\size + \trans, 2 * \size);  
 \coordinate (K) at (0 + \trans, \size);
 \coordinate (L) at (4 * \size + \trans, 0);
 \coordinate (M) at (3 * \size + \trans, \size);
 \coordinate (N) at (1/2 * \size + \trans, 1/2 * \size);
 
 \draw [very thick, color=green2] (G) -- (E) ;
 \draw [very thick, color=green2] (G) -- (N) ;
 \draw [very thick, color=green2] (G) -- (M) ;
 \draw [thin, color=black] (G) -- (F) ;
 \node [black] at (\size+ \trans, \size) {\tiny$\bullet$}; 
 \node [black] at (1.5* \size+ \trans, \size) {\tiny$\bullet$}; 
 \draw [very thick] (B) -- (C);
 \draw [very thick] (D) -- (E)-- (A);
 \draw [very thick, dashed] (C) -- (D);
 \draw [very thick, dashed] (A) -- (B);

\end{tikzpicture}
\caption{The intersection of tropicalizations: $\Trop ((0,1)), \Trop ((1,0)), \Trop ((1,1)) $}\label{TwopointblowupFig}
\end{figure}
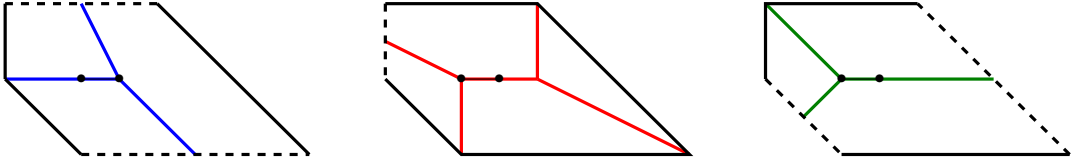
\end{example}

\begin{example}\label{Twopointblowup1}
Let $X$ be the two-point blowup of $\C\P^2$ determined by the moment polytope
$$
P =\{ (u_1, u_2) \in M_\R : u_1 \geq 0,\, u_2 \geq 0,\, 2 - u_1 + u_2 \geq 0,\, 5 - u_2 \geq 0,\, 1 + u_1 - u_2 \geq 0 \}.
$$

We begin by drawing the tropicalizations $\Trop \left(x_1  \frac{\pa \frak{PO}_0}{\pa x_1} \right)$ and $\Trop \left(x_2  \frac{\pa \frak{PO}_0}{\pa x_2} \right)$ in Figure~\ref{TwopointblowupFig2} and observe that the intersection of the tropicalizations appears as
$$
\{ (u_1, u_2 ) \in M_\R : 2 \leq u_1 \leq 4, \, u_2 = u_1 - 0.5 \} \cup \{ (1, 1) \},
$$. 

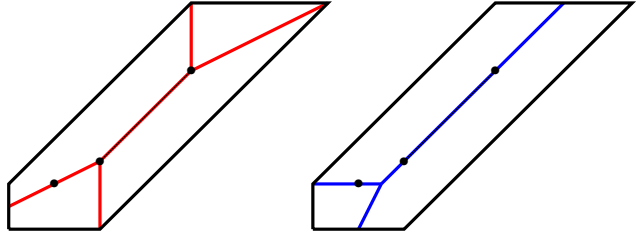
\begin{figure}[h]
\centering
\begin{tikzpicture}
 \def\size{0.6}  
 \def\trans{4}

 \coordinate (A) at (0, 0);
 \coordinate (B) at (\size*2, 0);
 \coordinate (C) at (\size*7, \size*5);
 \coordinate (D) at (\size*4, \size*5);
 \coordinate (E) at (0, \size*1);
 \coordinate (F) at (\size*1, \size*1);
 \coordinate (G) at (\size*2, \size*1.5);
 \coordinate (H) at (\size*4, \size*3.5);
 \coordinate (I) at (0, \size*0.5);
 \coordinate (J) at (\size*5.5, \size*5);  
 \coordinate (K) at (\size*1.5, \size*1);
 \coordinate (L) at (\size*1, 0);

 \draw [very thick, color=blue] (K) -- (J) ;
 \draw [very thick, color=blue] (K) -- (E) ;
 \draw [very thick, color=blue] (K) -- (L) ;
 \draw [thin, color=black] (G) -- (H) ;
 \node [black] at (\size, \size) {\tiny$\bullet$}; 
 \node [black] at (G) {\tiny$\bullet$}; 
 \node [black] at (H) {\tiny$\bullet$};  
 \draw [very thick] (A) -- (B) -- (C) -- (D)-- (E)-- (A);

 \coordinate (A) at (0 - \trans, 0);
 \coordinate (B) at (\size*2 - \trans, 0);
 \coordinate (C) at (\size*7 - \trans, \size*5);
 \coordinate (D) at (\size*4 - \trans, \size*5);
 \coordinate (E) at (0 - \trans, \size*1);
 \coordinate (F) at (\size*1 - \trans, \size*1);
 \coordinate (G) at (\size*2 - \trans, \size*1.5);
 \coordinate (H) at (\size*4 - \trans, \size*3.5);
 \coordinate (I) at (0 - \trans, \size*0.5);
 \coordinate (J) at (\size*5.5 - \trans, \size*5);  
 \coordinate (K) at (\size*1.5 - \trans, \size*1);
 \coordinate (L) at (\size*1 - \trans, 0);

 \draw [very thick, color=red] (G) -- (H) ;
 \draw [very thick, color=red] (H) -- (C) ;
 \draw [very thick, color=red] (H) -- (D) ;
 \draw [very thick, color=red] (G) -- (B) ;
 \draw [very thick, color=red] (G) -- (I) ;
 \draw [thin, color=black] (G) -- (H) ;
 \node [black] at (\size - \trans, \size) {\tiny$\bullet$}; 
 \node [black] at (G) {\tiny$\bullet$}; 
 \node [black] at (H) {\tiny$\bullet$};  
 \draw [very thick] (A) -- (B) -- (C) -- (D)-- (E)-- (A);

\end{tikzpicture}
\caption{The intersection of tropicalizations: $\Trop \left(x_1\frac{\pa \frak{PO}_0}{\pa x_1}\right), \Trop \left(x_2\frac{\pa \frak{PO}_0}{\pa x_2}\right)$}\label{TwopointblowupFig2}
\end{figure}

By using the method of probes developed by McDuff \cite{Mc}, one can see however that the intersection contains positions of displaceable fibers. This example shows that two tropicalizations $\Trop \left(x_1\frac{\pa \frak{PO}_0}{\pa x_1}\right)$ and  $\Trop \left(x_2\frac{\pa \frak{PO}_0}{\pa x_2}\right)$ are not sufficient to detect non-displaceable toric fibers precisely. In fact, $L((1,1))$ and $L((3, 2.5))$ are the only non-displaceable fibers. So, more irrelevant positions must be eliminated and Theorem~\ref{THEOREMA} exactly tells us which fibers must be filtered out. In this case, we take the intersection of three tropicalizations
$$
\Trop ((1,0))
\cap \, \Trop ((0,1)) \cap \Trop ((1,-1)) \cap \textup{ Int} (P)
$$
and observe that the intersection points are exactly $(1,1)$ and $(3, 2.5)$ (See Figure~\ref{TwopointblowupFig3}). These are exactly the positions of non-displaceable fibers of $X$. 
\begin{figure}[h]
\centering
\begin{tikzpicture}
 \def\size{0.6}  
 \def\trans{4}

 \coordinate (A) at (0, 0);
 \coordinate (B) at (\size*2, 0);
 \coordinate (C) at (\size*7, \size*5);
 \coordinate (D) at (\size*4, \size*5);
 \coordinate (E) at (0, \size*1);
 \coordinate (F) at (\size*1, \size*1);
 \coordinate (G) at (\size*2, \size*1.5);
 \coordinate (H) at (\size*4, \size*3.5);
 \coordinate (I) at (0, \size*0.5);
 \coordinate (J) at (\size*5.5, \size*5);  
 \coordinate (K) at (\size*1.5, \size*1);
 \coordinate (L) at (\size*1, 0);

 \draw [very thick, color=blue] (K) -- (J) ;
 \draw [very thick, color=blue] (K) -- (E) ;
 \draw [very thick, color=blue] (K) -- (L) ;
 \node [black] at (\size, \size) {\tiny$\bullet$}; 
 \node [black] at (\size*3, \size*2.5) {\tiny$\bullet$};  
 \draw [very thick] (B) -- (C);
 \draw [very thick] (D)-- (E)-- (A);
 \draw [very thick, dashed] (A) -- (B);
 \draw [very thick, dashed] (C) -- (D);

 \coordinate (A) at (0 - \trans, 0);
 \coordinate (B) at (\size*2 - \trans, 0);
 \coordinate (C) at (\size*7 - \trans, \size*5);
 \coordinate (D) at (\size*4 - \trans, \size*5);
 \coordinate (E) at (0 - \trans, \size*1);
 \coordinate (F) at (\size*1 - \trans, \size*1);
 \coordinate (G) at (\size*2 - \trans, \size*1.5);
 \coordinate (H) at (\size*4 - \trans, \size*3.5);
 \coordinate (I) at (0 - \trans, \size*0.5);
 \coordinate (J) at (\size*5.5 - \trans, \size*5);  
 \coordinate (K) at (\size*1.5 - \trans, \size*1);
 \coordinate (L) at (\size*1 - \trans, 0);

 \draw [very thick, color=red] (G) -- (H) ;
 \draw [very thick, color=red] (H) -- (C) ;
 \draw [very thick, color=red] (H) -- (D) ;
 \draw [very thick, color=red] (G) -- (B) ;
 \draw [very thick, color=red] (G) -- (I) ;
 \node [black] at (\size-\trans, \size) {\tiny$\bullet$}; 
 \node [black] at (\size*3-\trans, \size*2.5) {\tiny$\bullet$};  
 \draw [very thick] (A) -- (B) -- (C) -- (D)-- (E);
 \draw [very thick, dashed] (E)-- (A);

 \coordinate (A) at (0 + \trans, 0);
 \coordinate (B) at (\size*2 + \trans, 0);
 \coordinate (C) at (\size*7 + \trans, \size*5);
 \coordinate (D) at (\size*4 + \trans, \size*5);
 \coordinate (E) at (0 + \trans, \size*1);
 \coordinate (F) at (\size*1 + \trans, \size*1);
 \coordinate (G) at (\size*2 + \trans, \size*1.5);
 \coordinate (H) at (\size*4 + \trans, \size*3.5);
 \coordinate (I) at (0 + \trans, \size*0.5);
 \coordinate (J) at (\size*5.5 + \trans, \size*5);  
 \coordinate (K) at (\size*1.5 + \trans, \size*1);
 \coordinate (L) at (\size*1 + \trans, 0);
 \coordinate (M) at (\size*2.5 + \trans, \size*2.5);
 \coordinate (N) at (\size*4.5 + \trans, \size*2.5);
 \coordinate (O) at (\size*2 + \trans, \size*3); 

 \draw [very thick, color=green2] (A) -- (M) ;
 \draw [very thick, color=green2] (M) -- (N) ;
 \draw [very thick, color=green2] (M) -- (O) ; 
 \node [black] at (\size +\trans, \size) {\tiny$\bullet$}; 
 \node [black] at (\size*3 +\trans, \size*2.5) {\tiny$\bullet$};  
 \draw [very thick] (E) -- (A) -- (B);
 \draw [very thick] (C) -- (D);
 \draw [very thick, dashed] (B) -- (C);
 \draw [very thick, dashed] (D)-- (E);

\end{tikzpicture}
\caption{The intersection of tropicalizations: $\Trop ((1,0)), \Trop ((0,1)), \Trop ((1,-1))$}\label{TwopointblowupFig3}
\end{figure}
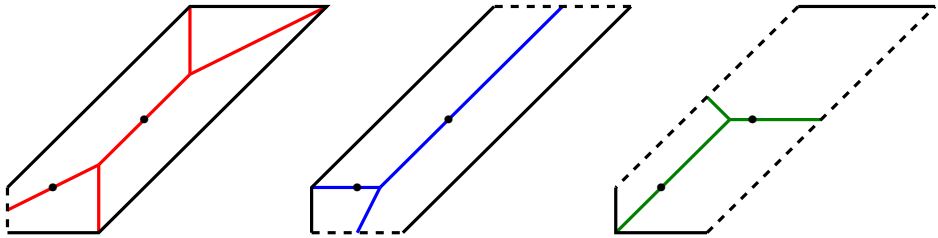
\end{example}

\section{Bulk-deformations and potential functions}\label{section:Bulk}
The aim of this section is to review notations and results from \cite{FOOOToric2} which will be used in the proofs of main Theorems (Section~\ref{ProofA} and Section~\ref{ProofB}). Again, throughout this section, let $X$ be the compact symplectic toric manifold determined by the image of a moment map $\pi : X \to P \subset M_\R$. As in ~\eqref{descriptionofpolytope}, the moment polytope $P$ has the unique description 
$$
\displaystyle P = \bigcap_{j=1}^m \left\{ \u \in M_\R : l_j (\u) \geq 0 \right\}.
$$
where $l_j (\u) := \langle \u, \v_j \rangle - \lambda_j$. In this description, each hyperplane given by $l_j = 0$ in $M_\R$ must contain one single facet $P_j$ of $P$. For an integer $j \in \{1, \cdots, m\}$, the divisor $\scr{D}_j$ of $X$ is declared to be $\pi^{-1}(P_j)$, which is a component of the toric divisor of $X$.  Let $\scr{A}(\Lambda_0)$ and $\scr{A}(\Lambda_+)$ be respectively the free $\Lambda_0$-module and $\Lambda_+$-module generated by all $\scr{D}_j$'s. In \cite{FOOOToric2}, Fukaya-Oh-Ohta-Ono employed these ambient cycles on $X$ to deform the given $A_\infty$-algebra $\left( H^\bullet \left( L(\u); \Lambda_0 \right), \{ \frak{m}^k : k \geq 0 \} \right)$ defined in Section~\ref{Ainftyalgebrasontoricfibers}. For every $\frak{b} \in \scr{A}(\Lambda_0)$, one can obtain the deformed $A_\infty$-algebra $\left( H^\bullet \left( L(\u); \Lambda_0 \right), \{ \frak{m}^k_{\frak{b}} : k \geq 0 \} \right)$. Such deformation is called a \emph{bulk-deformation} of the $A_\infty$-algebra. We will not spell out how an ambient cycle deforms the structure maps in the $A_\infty$-algebra (we refer it to \cite{FOOOToric2} and \cite{FOOOSurv}) because it is sufficient to observe how a bulk-deformation affects the potential function of the $A_\infty$-algebra for our purpose detecting non-displaceable toric fibers with a non-zero deformed Floer cohomology. With the emphasis on this aspect, following \cite{FOOOToric2}, we will recall relation between bulk-deformations and potential functions in this section. 

For each $\frak{b} \in \scr{A}(\Lambda_0)$ and each toric fiber $L(\u)$ of $X$, every cochain $b \in H^1(L(\u); \Lambda_0)$ is turned out to be a bounding cochain meaning that $\sum_{k=0}^\infty \frak{m}^k_{\frak{b}} (b^{\otimes k})$ is a multiple of the Poincar\'{e} dual $\textup{PD}[L(\u)]$ (Theorem 8.2 in \cite{FOOOSurv}). Therefore, for each $\frak{b} \in \scr{A}(\Lambda_0)$ and $\u \in \text{Int}(P)$, the potential function $\frak{PO}^\u_\frak{b}$ can be defined on $H^1 (L(\u); \Lambda_0)$. Moreover, for the further deformed $A_\infty$-algebra $\{ \frak{m}^k_{\frak{b}, b} : k \geq 0 \}$ by $b$ as in Section~\ref{Ainftyalgebrasontoricfibers}, it follows that the deformed differential $\frak{m}^1_{\frak{b},b}$ satisfies $\frak{m}^1_{\frak{b},b} \circ \frak{m}^1_{\frak{b},b} = 0$ from the $A_\infty$-relation and the fact that $\textup{PD}[L(\u)]$ is the strict unit in the deformed $A_\infty$-algebra $\left(H^\bullet(L(\u); \Lambda_0), \{ \frak{m}^k_{\frak{b}} : k \geq 0 \}\right)$. We then define the \emph{\ulp bulk-\urp deformed Floer cohomology} given by the differential $\frak{m}^1_{\frak{b},b}$ as follows:
$$
HF^\bullet((L(\u), \frak{b}, b), (L(\u), \frak{b}, b); \Lambda_0) := \Ker \, (\frak{m}^1_{\frak{b},b}) /  \Im  \, (\frak{m}^1_{\frak{b},b}). 
$$ 

\subsection{Leading term equations}\label{subsectionleadingtermequation}
For any point $\u$ in the interior of a moment polytope $P$, we arrange the values $l_j(\u) \, (j= 1, \cdots, m)$ into the ascending order, and we denote the arranged values by $S_l$ so that $S_l$'s are real numbers obeying two conditions
$$
\{ S_l : l = 1, \cdots, m(\u) \} = \{ l_j (\u) : j = 1, \cdots, m \} \quad \textup{and} \quad 0 < S_1 < S_2 < \cdots < S_{m(\u)}.
$$ 
Here, an integer $m(\u)$ denotes the number of the different values of $l_j(\u)$'s and it certainly depends on $\u$. The number $S_l$ is frequently called the \emph{energy level} because $2 \pi \cdot S_l$ is the symplectic area of a certain holomorphic disc whose boundary maps into $L(\u)$. 

In order to take an index system reflecting the energy level $S_l$'s, it is convenient to replace $j$ with $(r,s)$ determined by the following lexicographic order: 
\begin{enumerate}
\item Take $r = l$ so that $l_j(\u) = S_l$. 
\item Fix a numbering of the indices $i$'s with $l_i(\u) = S_l$ from $1$ to $a_l$ where an integer $a_l$ is the number of $l_i$'s such that $l_i(\u)= S_l$ and take $s$ as the corresponding number of the given index $j$. 
\end{enumerate}

We identify two index systems $\{j\}$ and $\{(r,s)\}$ so that for each $j$ there exists $(r,s)$ such that $j = (r,s)$ and vice versa. When we need to consider two different quantities or vectors indexed by $j$ and $r,s$ even though $j = (r,s)$, we will use $r,s$ \emph{without} parenthesis to emphasize that it is meaningful as a double index and is not a substitution of the number $j$. For example, a vector $\v_{(r,s)}$ is identical with $\v_j$ whenever $j = (r,s)$, however a vector $\v_{r,s}$ is not necessarily equal to $\v_j$ even though $j = (r,s)$.

Following Section 4 in \cite{FOOOToric2}, we now attempt to simplify the potential function by introducing a new coordinate system compatible with the energy level $S_l$'s. Let $A_l^\perp$ be the real vector space generated by the inward primitive normal vectors to facets up to the level $S_l$, that is,
\begin{equation}\label{Aperp}
A_l^\perp := \langle \v_{(1,1)}, \cdots, \v_{(1,a_1)}, \cdots, \v_{(l,1)}, \cdots, \v_{(l,a_l)} \rangle \subset N_\R.
\end{equation}
Setting $\dim_\R A_0^\perp := 0$, let $d_l := \dim_\R A_l^\perp - \dim_\R A_{l-1}^\perp$ and let $\kappa(\u)$ be the smallest integer $l$ such that $A_l^\perp = N_\R$.  In order to write down the potential function on $H^1(L(\u); \Z)$ as a Laurent series in terms of variable $\{ y_j \}$ in Section~\ref{LFT on toric}, we fixed a basis $\{e_1, \cdots, e_n\}$ of the lattice $M$ giving us an identification of $M \simeq \Z^n$ and a coordinate system on $H^1(L(\u); \Z)$, and we took $y_j := \exp ( e^*_j )$ where $\{ e_1^* , \cdots , e_n^* \}$ is the dual basis of $\{ e_1, \cdots, e_n \}$. Instead of $\{ e_1^* , \cdots , e_n^* \}$, we choose a basis of $N_\R$
\begin{equation}\label{ers*}
\{ e^*_{r,s} : 1 \leq r \leq \kappa(\u), \, 1 \leq s \leq d_r\}
\end{equation}
satisfying the following conditions:
\begin{enumerate}
\item For each $l \in \N$, $\{ e^*_{r,s} : 1 \leq r \leq l, \, 1 \leq s \leq d_r\}$ forms a $\Q$-basis of $A^\perp_l \cap N_\Q$.  
\item Each $\v_j$ is contained in $\displaystyle \bigoplus_{r=1}^{\kappa(\u)} \bigoplus_{s=1}^{d_r} \, \Z \, e^*_{r,s}$.
\end{enumerate}
Regarding $e^*_{r,s}$ in $N_\R \simeq \Hom(M_\R, \R)$ as a function on $M_\R$ and taking $y_{r,s} := \exp \left( e^*_{r,s} \right)$, we can express the potential function as a Laurent series in terms of variables $\{ y_{r,s}: 1 \leq r \leq \kappa(\u), \, 1 \leq s \leq d_r \}$. 

For later usage, we keep an explicit formula connecting $\{ y_j : j = 1, \cdots, m\}$ to $\{ y_{r,s}: 1 \leq r \leq \kappa(\u), \, 1 \leq s \leq d_r \}$. By the second condition of choosing $\{e^*_{r,s} \}$ and the fact that $\{ \v_j : j = 1, \cdots, m \}$ generates $N$ as a $\Z$-module, we obtain
$$
e_j^* = \sum_{r=1}^{\kappa(\u)} \sum_{s=1}^{d_r} a_{j}^{r,s} e_{r,s}^*
$$
for some $a_{j}^{r,s} \in \Z$. It leads to
\begin{equation} \label{coord}
y_j = \prod_{r=1}^{\kappa(\u)} \prod_{s=1}^{d_r} y_{r,s}^{a_{j}^{r,s}}. 
\end{equation}

Now, we focus on the finite sum of Laurent monomials of variables $\{ y_{r,s} \}$ forming the coefficient of $T^{S_{l}}$ in the leading order potential function $\frak{PO}^\u_0$ at $\u$ of $P$.\footnote{It becomes transparent in Lemma~\ref{potentialfunctionbulkdeformb+} why it deals with terms in the leading-order potential function (not in the potential function).} This sum is denoted by $(\frak{PO}^\u_0)_l$ and is written as follows:
$$
\left(\frak{PO}^\u_0\right)_l = \sum_{a=1}^{a_l} \y^{\v_{(l,a)}}
$$
where $a_l$ is the number of $l_j$'s such that $l_j(\u)= S_l$ and 
$$
\displaystyle \v_{(l,a)} = \sum_{r=1}^{\kappa(\u)} \sum_{s=1}^{d_r} {v}^{r,s}_{(l,a)} e_{r,s}^*, \quad \displaystyle \y^{\v_{(l,a)}} =\prod_{r=1}^{\kappa(\u)} \prod_{s=1}^{d_r} y_{r,s}^{{v}^{r,s}_{(l,a)}}.
$$ 
By the second condition of choosing $\{e^*_{r,s} \}$, for each $(l,a)$ we have
$$
\displaystyle \v_{(l,a)} \in \bigoplus_{r=1}^{l} \, \bigoplus_{s=1}^{d_r} \, \Z e^*_{r,s}
$$
and hence $(\frak{PO}^\u_0)_l$ is a Laurent polynomial with respect to $\{ y_{r,s} \}$, i.e.
\begin{equation}\label{wherePOul}
\left(\frak{PO}^\u_0\right)_l \in \Z \left[y_{1,1}, y^{-1}_{1,1}, \cdots, y_{l,d_l}, y^{-1}_{l,d_l} \right].
\end{equation}

\begin{definition} [See Section 4 in \cite{FOOOToric2}] \label{leadingtermequation}
The \emph{leading term equation} at $\u$ of $P$ with respect to $\{ y_{r,s} \}$ is defined to be the system of equations 
\begin{equation}\label{leadingterm}
\begin{cases}
\displaystyle y_{1,s} \frac{\pa \left( \frak{PO}^\u_0 \right)_1 } {\pa \, y_{1,s}} = 0 \quad \text{ for } s = 1, \cdots , d_1 \\
\displaystyle y_{2,s} \frac{\pa \left( \frak{PO}^\u_0 \right)_2 } {\pa \, y_{2,s}} = 0 \quad \text{ for } s = 1, \cdots , d_2 \\
\quad \quad \vdots \\
\displaystyle y_{\kappa, s} \frac{\pa \left( \frak{PO}^\u_0 \right)_{\kappa} } {\pa \, y_{\kappa,s}} = 0 \quad \text{ for } s = 1, \cdots , d_\kappa
\end{cases}
\end{equation}
where $\kappa := \kappa(\u)$.
\end{definition}

Note that leading term equations depend on our choice of basis $\{e^*_{r,s} \}$, but the leading term equation with respect to one coordinate system is related to that of another coordinate system by a coordinate change. 

\subsection{Bulk-deformations and non-displacement of toric fibers}\label{subsection:Bulk-deformations and non-displacement of toric fibers}

There are at least two reasons why a bulk-deformation is introduced for toric cases in \cite{FOOOToric2}. One reason is that a bulk-deformation by $\scr{A}(\Lambda_+)$ allows us to get rid of all terms other than $\sum_{l=1}^{\kappa(\u)} \left(\frak{PO}^\u_0\right)_l T^{S_l}$ in the potential function $\frak{PO}^\u$ of $X$.\footnote{Theorem~\ref{Ainftyrelationsum} (4) is needed for the elimination process} It is computationally preferable because the potential function is in general hard to compute. Another reason is that a bulk-deformation by $\scr{A}(\Lambda_0)$ can change coefficients in $\frak{PO}^\u_0$ so that when dealing with leading term equations it gives us flexibility in some examples including Example 11.7 in \cite{FOOOToric2}. It is proved that non-displacement of toric fibers is guaranteed if a bulk-deformed Floer cohomology does not vanish (See Section 8 in \cite{FOOOToric2}). We will actively use bulk-deformations by $\scr{A}(\Lambda_0)$ to take advantage of the flexibility in Section~\ref{ProofA}. 

We denote by $\frak{PO}_\frak{b}^\u$ the bulk-deformed potential function by $\frak{b} \in \scr{A}(\Lambda_0)$ (or $\scr{A}(\Lambda_+)$) at $\u$. For a suitable choice of $\frak{b} \in \scr{A}(\Lambda_+)$, the potential function $\frak{PO}^\u$ can be deformed as follows: 

\begin{lemma} [See Section 4 in \cite{FOOOToric2}] \label{potentialfunctionbulkdeformb+}  For the potential function $\frak{PO}^\u$ at $\u \in \textup{Int}(P)$ of a compact symplectic toric manifold $X$, there exists a bulk-parameter 
$$
\frak{b} = \sum_{j=1}^m \frak{b}_j \cdot \scr{D}_j = \sum_{r=1}^{\kappa(\u)} \sum_{s=1}^{a_r} \frak{b}_{(r,s)} \cdot \scr{D}_{(r,s)}  \in \scr{A}(\Lambda_+)
$$ 
such that the bulk-deformed potential function by $\frak{b}$ is
$$
\frak{PO}_\frak{b}^\u =  \sum_{l=1}^{\kappa(\u)}  \left( \sum_{a=1}^{a_l} \, \y^{\v_{(l,a)}} \right) T^{S_l} = \sum_{l=1}^{\kappa(\u)} \left(\frak{PO}^\u_0\right)_l T^{S_l}.
$$
\end{lemma}

We notice that the logarithmic derivative of $\frak{PO}_\frak{b}^\u$ with respect to $\{y_{r,s}\}$ is exactly the leading term equation in Definition~\ref{leadingtermequation}. One of main results in \cite{FOOOToric2} are stated as follows. 

\begin{theorem} [Theorem 4.5 in \cite{FOOOToric2}] \label{FOOOToric2Main1} For $\u \in \text{Int}(P)$, the followings are equivalent.
\begin{enumerate}
\item The leading term equation at $\u$ of $P$ in ~\eqref{leadingterm} admits a solution on $(\C^*)^n$.\footnote{A choice of variable $\{y_{r,s}\}$ does not matter because the leading term equation with respect to any system of variables has a solution as soon as the leading term equation with respect to one system of variables has (See Lemma~\ref{solutionandbasechange}).}
\item There exists a bulk-parameter $\frak{b} \in \scr{A}(\Lambda_+)$ such that $\frak{PO}_{\frak{b}}^\u$ has a critical point on $\left( \Lambda_U \right)^n$.
\item There exist a cochain $b \in H^1 (L(\u); \Lambda_0)$ and a bulk-parameter $\frak{b} \in \scr{A}(\Lambda_+)$ such that the deformed Floer cohomology $HF^\bullet((L(\u), (\frak{b}, b)), (L(\u), (\frak{b}, b)); \Lambda_0)$ is isomorphic to $H(T^n; \Lambda_0)$.
\end{enumerate}
\end{theorem}

\begin{corollary} [See Section 4 in \cite{FOOOToric2}]
If the leading term equation at $\u$ of $P$ admits a solution on $(\C^*)^n$, then the toric fiber $L(\u)$ is non-displaceable. 
\end{corollary}

\begin{remark} \label{remarkaboutbulkparameter} 
The statement about the equivalence $(1) \Leftrightarrow (2)$ in Theorem 4.5 \cite{FOOOToric2} is slightly different from that of Theorem~\ref{FOOOToric2Main1}. In \cite{FOOOToric2}, a bulk-parameter is chosen among the $\Lambda_+$-module $\scr{A}^{\text{all}}(\Lambda_+)$ generated by all possible intersections of components of toric divisors. Nonetheless, the proof of Theorem 4.5 in \cite{FOOOToric2} implies that if $\frak{PO}^\u_\frak{b}$ has a critical point at $\u$ for $\frak{b} \in \scr{A}^{\text{all}}(\Lambda_+)$, then there exists a bulk-parameter $\frak{b}^\prime \in \scr{A}(\Lambda_+)$ such that $\frak{PO}^\u_{\frak{b}^\prime}$ has a critical point at $\u$. Therefore, in toric cases, bulk-deformations by components of codimension $1$ from the toric divisor is sufficient to achieve a non-vanishing deformed Floer cohomology. 
\end{remark}

Just like balanced fibers (Definition~\ref{balanced}), to include more non-displaceable toric fibers, Fukaya-Oh-Ohta-Ono introduced a notion of \emph{bulk-balanced} fibers. 

\begin{definition} [cf. Definition 3.17 in \cite{FOOOToric2}] \label{bulkbalanced} 
Let $\Sigma \subset N$ be the normal fan of a moment polytope $P$. Let $X_\Sigma$ be the compact complex toric manifold given by the fan $\Sigma$. A toric fiber $L(\u)$ is called $\emph{bulk-balanced}$ if there exists a sequence of triples $(\omega^{(i)}, P^{(i)}, \u^{(i)})$ such that 
\begin{enumerate}
\item Each symplectic form $\omega^{(i)}$ is a torus-invariant symplectic form associated to $X_\Sigma$ and the sequence of symplectic forms $\omega^{(i)}$ converges to $\omega$.
\item The normal fan of each $P^{(i)}$ coincides with the normal fan $\Sigma$ of $P$ and the sequence of polytopes $P^{(i)}$ converges to $P$ with respect to the Hausdorff distance.
\item At each position $\u^{(i)} \in \text{Int} (P^{(i)})$, the leading term equation at $\u^{(i)}$ of $P^{(i)}$ in ~\eqref{leadingterm} admits a solution on $(\C^*)^n$ and the sequence of positions $\u^{(i)} \in \text{Int}(P^{(i)})$ converges to $\u \in \text{Int}(P)$. 
\end{enumerate}
\end{definition}

\begin{remark}
Definition~\ref{bulkbalanced} is different from Definition 3.17 in \cite{FOOOToric2}. But, two definitions are not essentially different because one can get rid of higher order parts by taking a another bulk-parameter $\frak{b}_i^\prime$ in $\scr{A}(\Lambda_+)$ if necessary. 
\end{remark}

We now see the effect of bulk-deformations by $\frak{b} \in \scr{A}(\Lambda_0)$ to the potential function. 

\begin{lemma} [See Section 11 in \cite{FOOOToric2}] For the potential function $\frak{PO}^\u$ at $\u$ of a compact symplectic toric manifold $X$, there exists a bulk-parameter 
$$
\frak{b} = \sum_{j=1}^m \frak{b}_j \cdot \scr{D}_j = \sum_{r=1}^{\kappa(\u)} \sum_{s=1}^{a_r} \frak{b}_{(r,s)} \cdot \scr{D}_{(r,s)}  \in \scr{A}(\Lambda_0)
$$ 
such that
$$
\frak{PO}_\frak{b}^\u =  \sum_{l=1}^{\kappa(\u)}  \left( \sum_{a=1}^{a_l} c_{(l,a)}  \cdot \y^{\v_{(l,a)}} \right) T^{S_l}
$$
where $c_{(l,a)} := \exp \left( \frak{b}_{(l,a), 0} \right)$ and $\frak{b}_{(l,a)} = \frak{b}_{(l,a),0} + \frak{b}_{(l,a),+}$ with $\frak{b}_{(l,a),0} \in \C$ and $\frak{b}_{(l,a),+} \in \Lambda_+$. 
\end{lemma}

For each $\frak{b} \in \scr{A}(\Lambda_0)$ and $1 \leq l \leq \kappa(\u)$, let  
\begin{equation}\label{llevelpotential}
\left(\frak{PO}_\frak{b}^\u \right)_l := \left( \sum_{a=1}^{a_l} \exp \left( \frak{b}_{(l,a), 0} \right) \cdot \y^{\v_{(l,a)}} \right).
\end{equation}
Next, we take the logarithmic derivative of $\left(\frak{PO}_\frak{b}^\u \right)_l$ with respect to $\{ y_{r,s} \}$ to obtain the \emph{generalized leading term equation}. 

\begin{definition} [See Section 11 in \cite{FOOOToric2}] \label{generalizedleadingtermequation} For $\frak{b} \in \scr{A}(\Lambda_0)$, the \emph{generalized leading term equation} at $\u$ of $P$ with respect to $\{ y_{r,s} \}$ is defined to be the system of equations
\begin{equation}\label{genleadingterm}
\begin{cases}
\displaystyle y_{1,s} \frac{\pa \left( \frak{PO}_\frak{b}^\u \right)_1 } {\pa \, y_{1,s}} = 0 \quad \text{ for } s = 1, \cdots , d_1 \\
\displaystyle y_{2,s} \frac{\pa \left( \frak{PO}_\frak{b}^\u \right)_2 } {\pa \, y_{2,s}} = 0 \quad \text{ for } s = 1, \cdots , d_2 \\
\quad \quad \vdots \\
\displaystyle y_{\kappa, s} \frac{\pa \left( \frak{PO}_\frak{b}^\u \right)_{\kappa} } {\pa \, y_{\kappa,s}} = 0 \quad \text{ for } s = 1, \cdots , d_\kappa
\end{cases}
\end{equation}
where $\kappa := \kappa(\u)$.
\end{definition}

Generalized leading term equations depend on our choices of a bulk parameter $\frak{b}$ and a basis $\{e^*_{r,s} \}$. Note that the coefficients in \eqref{genleadingterm} might change as $\frak{b}$ varies. Once $\frak{b}$ is fixed, they are independent up to coordinate changes. Furthermore, the following lemma asserts that a choice of $\{y_{r,s}\}$ does not matter when it comes to the existence of solutions. 

\begin{lemma} [Lemma 4.2 in \cite{FOOOToric2}] \label{solutionandbasechange} The system of equations
\begin{equation} \label{y_l,j}
y_{r,s} \frac{ \pa \frak{PO}_\frak{b}^\u  }{ \pa y_{r,s} } = 0
\end{equation}
has a solution $\{ ( y_{r,s} ) :  y_{r,s} \in \Lambda_U \}$ if and only if the system of equations
\begin{equation} \label{y_i}
y_i \frac{ \pa \frak{PO}_\frak{b}^\u }{ \pa y_i } = 0
\end{equation}
has a solution $\{ (y_i) : y_i  \in \Lambda_U \}$. 
\end{lemma}

We remark that the number of solutions in \eqref{y_l,j} might be different from that in \eqref{y_i}. Based on Lemma~\ref{solutionandbasechange}, in the sense of the existence of solutions, two systems are interchangeably used. One of main results in \cite{FOOOToric2} are stated as follows. 

\begin{theorem} [See Section 11 in \cite{FOOOToric2}] \label{FOOOToric2Main} For $\u \in \text{Int}(P)$, the followings are equivalent.
\begin{enumerate}
\item A generalized leading term equation at $\u$ of $P$ in ~\eqref{genleadingterm} admits a solution $\{ (y_{r,s}) : y_{r,s} \in \C^* \}$ for some $\frak{b} \in \scr{A}(\Lambda_0)$.
\item There exists a bulk-parameter $\frak{b}^\prime = \sum_{j=1}^m \frak{b}_j^\prime \scr{D}_j \in \scr{A}(\Lambda_0)$ such that $\frak{b}_j - \frak{b}^\prime_j \in \Lambda_+$ for all $j$ and $\frak{PO}^\u_{\frak{b}^\prime}$ has a critical point on $\left( \Lambda_U \right)^n$.
\item There exist a cochain $b \in H^1 (L(\u); \Lambda_0)$ and a bulk-parameter $\frak{b}^\prime = \sum_{j=1}^m \frak{b}_j^\prime \scr{D}_j \in \scr{A}(\Lambda_0)$ such that $\frak{b}_j - \frak{b}^\prime_j \in \Lambda_+$ for all $j$ and the bulk-deformed Floer cohomology $HF((L(\u), (\frak{b}^\prime, b)), (L(\u), (\frak{b}^\prime, b)); \Lambda_0)$ is isomorphic to $H(T^n; \Lambda_0)$. 
\end{enumerate}
\end{theorem}

\begin{corollary} [See Section 11 in \cite{FOOOToric2}] \label{FOOOToric2MainCor} If a generalized leading term equation of $\u$ of $P$ admits a solution on $\C^*$ for some $\frak{b} \in \scr{A}(\Lambda_0)$, the toric fiber $L(\u)$ is non-displaceable. 
\end{corollary}

\begin{definition} [Definition 9.7 in \cite{FOOOSurv}] \label{stronglybulkbalanced} 
A toric fiber $L(\u)$ is called \emph{strongly bulk-balanced}\footnote{Here, we stick to use the terminology from \cite{FOOOSurv}. However, the authors do not know whether a strongly bulk-balanced fiber is bulk-balanced or not. In this sense, the terminology might have a chance to be misleading.} if a generalized leading term equation at $\u$ admits a solution on $\left( \C^* \right)^n$ for some $\frak{b} \in \scr{A}(\Lambda_0)$ as in Theorem~\ref{FOOOToric2Main}. 
\end{definition}

\section{Proof of Theorem~\ref{THEOREMA}}\label{ProofA}

The goal of this section is to prove the main theorem of this paper. 

\begin{theorem}[Theorem~\ref{THEOREMA}] \label{maintheroem1}
Let $X$ be the compact symplectic toric manifold determined by a moment polytope $P$. For a point $\u$ in the interior of $P$, the followings are equivalent:
\begin{enumerate}
\item The fiber $L(\u)$ is strongly bulk-balanced (See Definition~\ref{stronglybulkbalanced}).
\item The point $\u$ is contained in the intersection of tropicalizations $\, \Trop(P, \m)$ relative to $\m$ over all lattice points $\m \in M$. 
\item The point $\u$ is contained in the intersection of tropicalizations  $\, \Trop(P, \m)$ relative to $\m$ over all primitive lattice $\m \in M$ that is orthogonal to an $(n-1)$-dimensional subspace generated by facet normal vectors.  
\end{enumerate}
\end{theorem}

By Corollary~\ref{FOOOToric2MainCor} and Theorem~\ref{maintheroem1}, we have the following Corollary. 

\begin{corollary} \label{maincorollary}
The intersection of tropicalizations $\, \Trop(P, \m)$ in Theorem~\ref{maintheroem1} (3) in the interior of $P$
 lifts to non-displaceable toric fibers.
\end{corollary}

We begin by proving the two lemmas.

\begin{lemma} \label{bulkimpliestropinter}
If the toric fiber $L(\u)$ over a point $\u \in \textup{Int}(P)$ is strongly bulk-balanced, then $\u$ is in $\Trop \left(P, \m \right)$ for any $\m \in M$. 
\end{lemma}

\begin{proof} 
Suppose that we are given a strongly bulk-balanced fiber $L(\u)$. We then consider the real vector space $A_l^\perp:= \langle \v_{(1,1)}, \cdots, \v_{(1,a_1)}, \cdots, \v_{(l,1)}, \cdots, \v_{(l,a_l)} \rangle$ defined in ~\eqref{Aperp} and the ascending chain of real vector spaces $A_0^\perp = \{0\} \subseteq A_1^\perp \subseteq \cdots \subseteq A_{\kappa(\u)}^\perp = N_\R$. For simplicity, let $\kappa := \kappa(\u)$.
For each space $A_l^\perp$, we have the real vector space $A_l$ orthogonal to $A_l^\perp $ given by
$$
A_l :=\{ \m \in M : \langle \m, \v \rangle = 0 \text{ for all } \v \in A_l^\perp\}.
$$ 
We then have the descending chain of real vectors space $M_\R = A_0 \supseteq A_1 \supseteq \cdots \supseteq A_\kappa = \{ 0 \}$. 

If $\m$ is the zero vector in $M$, then $\Trop (\m) = M_\R$ so that the position $\u$ is obviously in $\Trop  (\m)$. For any nonzero lattice point $\m$, there exists a unique integer $\nu$ with $0 \leq \nu \leq \kappa-1$ such that $\m$ is contained in $A_{\nu} \backslash A_{\nu+1}$. Since $\m \in A_\nu$, we have
\begin{equation}\label{lemmaequation1}
\langle \m, \v_{(1,1)} \rangle = \cdots = \langle \m, \v_{(1, a_1)} \rangle = \cdots = \langle \m, \v_{(\nu, 1)} \rangle = \cdots = \langle \m, \v_{(\nu, a_\nu)} \rangle = 0,
\end{equation}
and thus all equations $l_{(r,s)}$'s ranging over $ 1 \leq r \leq \nu$ and $1 \leq s \leq a_r$ are \emph{not} involved in the tropicalization $\Trop \left(\m \right)$. Since $\m \notin A_{\nu+1}$, there exists at least one $s$ with $1 \leq s \leq a_{\nu+1}$ such that $\langle \m, \v_{(\nu+1, s)} \rangle \neq 0$.

We claim further that there exist at least two integers $s_1$ and $s_2$ with $1 \leq s_1 < s_2 \leq a_{\nu+1}$ such that $\langle \m, \v_{(\nu+1, s_1)} \rangle \neq 0$ and $\langle \m, \v_{(\nu+1, s_2)} \rangle \neq 0$ once the existence of nonzero solution of the generalized leading term equation ~\eqref{genleadingterm} is assumed. For a contradiction, suppose that $\langle \m, \v_{(\nu+1, 1)} \rangle \neq 0$ and $\langle \m, \v_{(\nu+1,s)} \rangle = 0$ for all $s > 1$ by changing a numbering of $s$ if necessary. We shall find a coordinate system $\{ y_{r,s} \}$ that does \emph{not} admit a nonzero solution in order to establish a contradiction to Lemma~\ref{solutionandbasechange} saying that the generalized leading term equation must have a solution on $(\C^*)^n$ regardless of choices of coordiante system. 

We now take a basis $\{e_{r,s}^* \}$ of $N_\Q$ as in Section~\ref{subsectionleadingtermequation} by requiring the following conditions:
\begin{enumerate}
\item For each $l$, $\left\{ e^*_{r,s} : 1 \leq r \leq l, 1 \leq s \leq d_r \right\}$ forms a $\Q$-basis of $A_l^\perp \cap N_\Q$. 
\item Each $\v_j$ is in $\displaystyle \bigoplus_{r=1}^{\kappa} \bigoplus_{s=1}^{d_r} \, \Z \, e^*_{r,s}$.
\end{enumerate} 
We additionally require that $\left\{ e^*_{\nu+1,s} : 2 \leq s \leq d_r \right\}$ is perpendicular to the lattice point $\m$.  More precisely, the additional conditions are as follows:
\begin{enumerate}
\setcounter{enumi}{2}
\item Each $e^{*}_{\nu+1,s}$ with $s \geq 2$ is contained in $\{ \v \in A_{\nu+1}^\perp : \langle \m, \v \rangle = 0 \}$.
\item The vectors $\displaystyle \v_{(\nu+1,2)}, \cdots, \v_{(\nu+1, a_{\nu+1})}$ are in $\displaystyle \bigoplus_{r=1}^{\nu} \bigoplus_{s=1}^{d_r} \Z \, e^*_{r,s} \oplus \bigoplus_{s=2}^{d_{\nu+1}} \Z \, e^*_{\nu+1,s}$.
\end{enumerate}
Since the quotient space of $A_{\nu + 1}^\perp$ by $\left(A_{\nu}^\perp + \langle \v_{(\nu+1, 2)}, \cdots, \v_{(\nu+1, a_{\nu+1})} \rangle \right)$ is one-dimensional by our supposition, such a basis $\{e^*_{r,s} \}$ exists.

From this choice of $\{ e^*_{r,s} \}$, we obtain the corresponding coordinate system $\{y_{r,s} \}$. By the condition (4), a Laurent monomial 
$$
\displaystyle \y^{\v_{(\nu+1, a)}} :=  \prod_{r=1}^{\kappa} \prod_{s=1}^{d_r} y_{r,s}^{{v}^{r,s}_{(\nu+1,a)}}, \quad \text{ where } \v_{(\nu+1,a)} = \sum_{r=1}^{\kappa} \sum_{s=1}^{d_r} {v}^{r,s}_{(\nu+1,a)} e_{r,s}^*.
$$ 
with $a \geq 2$ is expressed in terms of $\{y_{r,s}:  1 \leq r \leq l, 1 \leq s \leq d_r \}$ and $\{ y_{\nu+1, s} : s \geq 2\}$. Therefore, by ~\eqref{llevelpotential}, we observe that  
$$
y_{\nu+1,1} \frac{\pa \left( \frak{PO}_\frak{b}^\u \right)_{\nu+1} } {\pa \, y_{\nu+1,1}}
$$
is a Laurent \emph{monomial}. Thus, it does not admit any nonzero solution and hence the generalizaed leading term equation does not have any solution in $(\C^*)^n$. The claim is now established. 

Consequently, we have at least two $l_{(l+1, s_1)}(\u)$ and $l_{(l+1, s_2)}(\u)$ having same value $S_{l+1}$ and moreover $S_{l+1}$ is indeed the minimum of $\left\{ l_{(r,s)}(\u) : \langle \m, \v_{(r,s)} \rangle \neq 0 \right\}$ because of ~\eqref{lemmaequation1}. Hence, $\u$ is in $\Trop (\m)$ by Proposition~\ref{descriptionoftropcurve}.

\end{proof}

\begin{lemma}\label{2ndlemmafortheoremA}
Let $P$ be an $n$-dimensional polytope in $M_\R \simeq \R^n$. If $\u \notin \Trop(P, \m)$ for some lattice point $\m$, there exists a primitive lattice point $\widetilde{\m}$ orthogonal to an $(n-1)$-dimensional subspace generated by facet normal vectors such that $\u \notin \Trop(P, \widetilde{\m})$
\end{lemma}

\begin{proof} Suppose that we are given $\u \notin \Trop(\m)$. We may assume that $l_{1}, \cdots, l_{s}$ are equations that contribute to the tropicalization $\Trop (\m)$ after renumbering the defining equations if necessary. Namely, $\left\{ l_j: \langle \m, \v_{j} \rangle \neq 0 \right\} = \{ l_{1}, \cdots, l_{s} \}$. By Proposition~\ref{descriptionoftropcurve}, there exists an integer $\nu$ with $1 \leq \nu \leq s$ such that the minimum of $ \left\{ l_j (\u): \langle \m, \v_{j} \rangle \neq 0 \right\} $ is attained only by $l_{\nu}(\u)$. Without any loss of generality, we may assume that $\nu = 1$.

If $\m$ is orthogonal to an $(n-1)$-dimensional subspace generated by facet normal vectors, then take a primitive vector $\widetilde{\m}$ of the direction of $\m$. Otherwise, letting $A^\perp_{>s} := \langle \v_{s+1}, \cdots, \v_m \rangle$, we extend the space $A^\perp_{>s}$ to an $(n-1)$-dimensional space not containing $\v_1$ by adding generators $\v_{j_1}, \cdots, \v_{j_r}$ with $1 < j_1 < \cdots < j_r \leq s$. Since $P$ is an $n$-dimensional polytope in $\R^n$, $\{ \v_j : 1 \leq j \leq m \}$ generates the whole space $\R^n$ so that such an extension exists. Choose a primitive vector $\widetilde{\m}$ normal to the space. 

By our choice of $\widetilde{\m}$, it is perpendicular to the $(n-1)$-dimensional space generated by the facet normals $\{ \v_{j_1}, \cdots, \v_{j_r}, \v_{s+1}, \cdots, \v_m \}$. Since
$$
l_1 \in \left\{ l_j: \langle \widetilde{\m}, \v_{j} \rangle \neq 0 \right\} \subset \left\{ l_j: \langle \m, \v_{j} \rangle \neq 0 \right\},
$$
we see that the minimum of $ \left\{ l_j (\u): \langle \widetilde{\m}, \v_{j} \rangle \neq 0 \right\} $ is achieved only by $l_{1}(\u)$. By Proposition~\ref{descriptionoftropcurve}, $\u \notin \Trop(\widetilde{\m})$. 
\end{proof}

Now, we start a proof of Theorem~\ref{maintheroem1}. 

\begin{proof}(of Theorem~\ref{maintheroem1})
Lemma~\ref{bulkimpliestropinter} proves $(1) \Rightarrow (2)$ and Lemma~\ref{2ndlemmafortheoremA} yields $(3) \Rightarrow (2)$. It is obvious that $(2) \Rightarrow (3)$. Thus, it remains to show that (2) implies (1). 

Suppose that we are given a position $\u \in \textup{ Int} (P)$ satisfying 
$$
\u \in \bigcap_{\m \in M} \Trop \left(\m \right).
$$
We now try to find a coordinate system $\{y_{r,s}\}$ making the leading term equation simple so that it contains many terms having a single variable factor. For each integer $l$ with $d_{l} := \lim A^\perp_{l} - \lim A^\perp_{l-1} > 0$, by rearranging $s$ in the $(l+1)$-level if necessary, we may assume 
\begin{align*}
A^\perp_{l-1} &= \langle \v_{(1,1)}, \cdots, \v_{(l-1,a_{l-1})} \rangle \subsetneq \langle \v_{(1,1)}, \cdots, \v_{(l-1,a_{l-1})}, \v_{(l, 1)} \rangle  \\
&\subsetneq \langle \v_{(1,1)}, \cdots, \v_{(l-1,a_{l-1})}, \v_{(l, 1)}, \v_{(l,2)} \rangle \subsetneq \cdots \\
&\subsetneq \langle \v_{(1,1)}, \cdots, \v_{(l-1,a_{l-1})}, \v_{(l, 1)}, \v_{(l,2)}, \cdots, \v_{(l, d_{l})} \rangle = A^\perp_l.
\end{align*}

We take a basis $\{e_{r,s}^* \}$ of $N_\Q$ as in Section~\ref{subsectionleadingtermequation} by requiring the following conditions:
\begin{enumerate}
\item For each $l$, $\left\{ e^*_{r,s} : 1 \leq r \leq l, 1 \leq s \leq d_r \right\}$ forms a $\Q$-basis of $A_l^\perp \cap N_\Q$. 
\item Each $\v_j$ is in $\displaystyle \bigoplus_{r=1}^{\kappa} \bigoplus_{s=1}^{d_r} \, \Z \, e^*_{r,s}$.
\end{enumerate} 
Additionally the following condition is required: 
\begin{enumerate}
\setcounter{enumi}{2}
\item $\v_{(l,s)} \in \N \, e^{*}_{l,s}$ for $1 \leq s \leq d_{l}$.
\end{enumerate}

By our choice of $\left\{ e^*_{r,s} \right\}$, we obtain the corresponding coordinate system $\{y_{r,s} \}$ and observe
$$
\left( \frak{PO}^\u_0 \right)_l = y_{l,1}^{n_{l,1}} + \cdots + y_{l,d_l}^{n_{l,d_l} }+ \frak{P}_l( y_{1,1}, \cdots , y_{l,d_l} )
$$
where $\v_{(l,s)} = n_{l,s} e^*_{l,s}$ for some $n_{l,s} \in \N$ and  $\frak{P}_l( y_{1,1}, \cdots , y_{l,d_l} )$ is a Laurent polynomial in terms of variables $y_{1,1}, \cdots , y_{l,d_l}$. 

Taking a bulk deformation given by
$$
\frak{b} := \sum_{r=1}^{\kappa} \sum_{s=1}^{d_r} \frak{b}_{(r,s)} \cdot \scr{D}_{(r,s)}
$$
with $\frak{b}_{(r,s)} \in \Lambda_0$, due to ~\eqref{llevelpotential}, we obtain
$$
\left( \frak{PO}^\u_\frak{b} \right)_l = c_{(l,1)} y_{l,1}^{n_{l,1}} + \cdots + c_{(l,d_l)} y_{l,d_l}^{n_{l,d_l} }+ \frak{P}_l( y_{1,1}, \cdots , y_{l,d_l} ).
$$
Keep in mind that each $c_{(r,s)} := \exp\left( \frak{b}_{(r,s),0} \right) \in \C^*$ can be chosen arbitrary by modifying $\frak{b}_{(r,s),0}$ where $\frak{b}_{(r,s)} = \frak{b}_{(r,s), 0} + \frak{b}_{(r,s), +}$ with $\frak{b}_{(r,s), 0} \in \C$ and $\frak{b}_{(r,s), +} \in \Lambda_+$. With respect to $\frak{b}$ and $\{y_{r,s}\}$, the generalized leading term equation ~\eqref{genleadingterm} is of the form:
\begin{equation}\label{generalizedleadingtermequationbecomes}
\begin{cases}
\displaystyle y_{1,s} \frac{\pa \left( \frak{PO}^\u_\frak{b} \right)_1 } {\pa \, y_{1,s}} = c_{(1,s)} n_{1,s} \, y_{1,s}^{ n_{1,s}} + y_{1,s} \frac{\pa \, \frak{P}_1 ( y_{1,1}, \cdots , y_{1,d_1} )}{\pa \, y_{1,s}} = 0 \quad \text{ for } s = 1, \cdots , d_1 \\
\displaystyle y_{2,s} \frac{\pa \left( \frak{PO}^\u_\frak{b} \right)_2 } {\pa \, y_{2,s}} = c_{(2,s)} n_{2,s} \, y_{2,s}^{ n_{2,s}} + y_{2,s} \frac{\pa \, \frak{P}_2 ( y_{1,1}, \cdots , y_{2,d_2} )}{\pa \, y_{2,s}} = 0 \quad \text{ for } s = 1, \cdots , d_2 \\
\quad \quad \vdots \\
\displaystyle y_{\kappa, s} \frac{\pa \left( \frak{PO}^\u_\frak{b} \right)_{\kappa} } {\pa \, y_{\kappa,s}} = c_{(\kappa,s)} n_{\kappa,s} \, y_{\kappa,s}^{ n_{\kappa,s}} + y_{\kappa,s} \frac{\pa \, \frak{P}_\kappa ( y_{1,1}, \cdots , y_{\kappa,d_\kappa} )}{\pa \, y_{\kappa,s}} = 0 \quad \text{ for } s = 1, \cdots , d_\kappa
\end{cases}
\end{equation}

We claim that the assumption $\u \in \bigcap_{\m \in M} \Trop \left(\m \right)$ implies that each
$$
y_{r,s} \frac{\pa \, \frak{P}_r ( y_{1,1}, \cdots , y_{r,d_r} )}{\pa \, y_{r,s}}
$$ 
contains at least one term. Suppose to the contrary that $y_{r,s} \frac{\pa \, \frak{P}_r ( y_{1,1}, \cdots , y_{r,d_r} )}{\pa \, y_{r,s}} \equiv 0$ for some $(r,s)$ with $1 \leq r \leq \kappa$ and $1 \leq s \leq d_r$. It yields that any term of $\frak{P}_r ( y_{1,1}, \cdots , y_{r,d_r} )$ must not contain any factors $y_{r,s}^{\pm 1}$. In other words, $\v_{(r,1)}, \cdots, \v_{(r, s-1)},$ $\v_{(r, s+1)},$  $\cdots$ $\v_{(r, a_r)}$ are contained in $\langle e^*_{1,1}, \cdots, e^*_{r,1} , \cdots , e^*_{r, s-1}, e^*_{r, s+1} , \cdots , e^*_{r, d_r} \rangle$. Since $\langle e^*_{1,1}, \cdots,  e^*_{r,1} , \cdots , e^*_{r, s-1}, e^*_{r, s+1} , \cdots , e^*_{r, d_r} \rangle$ has a dimension strictly less than $n$, we can take a lattice point $\m \in M$ such that 
\begin{enumerate}
\item $\m$ is perpendicular to $\langle \v_{(1,1)}, \cdots, \v_{(r-1, a_{r-1})}, \v_{(r, 1)}, \cdots, \v_{(r, s-1)}, \v_{(r, s+1)},  \cdots, \v_{(r, a_r)} \rangle$
\item $\m$ is not perpendicular to $\v_{(r,s)}$.
\end{enumerate}
We then see that $\Trop(\m)$ must \emph{not} contain $\u$ because at $\u$ the minimum is attained by a single equation $l_{(r,s)}$. It contradicts to the choice of $\u$ and hence the claim is asserted.

Now, by invoking the Baire category theorem and the above claim, we can take an $n$-tuple $\{ (\frak{y}_{r,s}) : \frak{y}_{r,s} \in \C^*, 1 \leq r \leq \kappa, \, 1 \leq s \leq d_r \}$ obeying
$$
\displaystyle \frak{y}_{r,s} \frac{\pa \, \frak{P}_r}{\pa \, y_{r,s}} ( \frak{y}_{1,1}, \cdots , \frak{y}_{r,d_r} ) \neq 0.
$$
Since $c_{(r,s)} := \exp\left( \frak{b}_{(r,s),0} \right) \in \C^*$ can be chosen independently, we can determine $c_{(r,s)}$ so that $\{ \frak{y}_{r,s} : 1 \leq r \leq \kappa, \, 1 \leq s \leq d_r \}$ is a solution of the equation ~\eqref{generalizedleadingtermequationbecomes}. Hence $L(\u)$ is strongly bulk-balanced.
\end{proof}

\section{Proof of Theorem~\ref{THEOREMB}}\label{ProofB}

In this section, we prove the following theorem by applying Theorem~\ref{THEOREMA}. 

\begin{theorem}[Theorem~\ref{THEOREMB}] \label{maintheorem2}
Let $X$ be the compact symplectic toric manifold determined by a moment polytope $P$. If the fiber $L(\u)$ over an interior point $\u$ of $P$ is bulk-balanced (See Definition~\ref{bulkbalanced}), then $L(\u)$ is strongly bulk-balanced (See Definition~\ref{stronglybulkbalanced}). 
\end{theorem}

We need two lemmas for proving Theorem~\ref{THEOREMB}. 

\begin{lemma}\label{convergenceoflambdaj}
Let $P$ be that the polytope with the description
$$
\left\{ l_j (\u) := \langle \u, \v_j \rangle - \lambda_j \geq 0 : j = 1, \cdots, m \right\}
$$ satisfying the requirements of ~\eqref{descriptionofpolytope}. We consider a sequence $\{ P^{(i)} \}$ of polytopes each of which has the description
$$
\left\{ l_j^{(i)}(\u) := \langle \u, \v_j \rangle - \lambda_j^{(i)} \geq 0: j = 1, \cdots, m \right\}
$$ 
satisfying the requirements of ~\eqref{descriptionofpolytope}. If the sequence $\left\{P^{(i)}\right\}$ converges to the polytope $P$ with respect to the Hausdorff distance (Definition~\ref{Hausdorffdistance}), then for each $j$, $\left(\lambda_j^{(i)} - \lambda_j\right)$ converges to $0$ as $i \to \infty$. 
\end{lemma}

\begin{proof}
Observe that the Hausdorff distance between $P^{(i)}$ and $P$ satisfies
$$
d_{\text{Haus}} \left( P, P^{(i)} \right) \geq \frac{\left| \lambda_j - \lambda_j^{(i)} \right|}{\| \v_j \|}
$$ 
where $\| \v_j \| = \sqrt{ (v_{j,1})^2 + \cdots (v_{j,n})^2 }$. Since $d_{\text{Haus}} \left( P, P^{(i)} \right) \to 0$ as $i \to \infty$, we obtain the desired conclusion.   
\end{proof}

\begin{lemma}\label{Hasudorffdistanceoftropicalization}
Let $Q$ be the polyhedron having the description
$$
\left\{ l_j (\u) := \langle \u, \v_j \rangle - \lambda_j \geq 0 : j = 1, \cdots, s \right\}
$$ 
that satisfies the requirements of ~\eqref{descriptionofpolytope}.
We consider a sequence $\{ Q^{(i)} \}$ of polyhedra determined by the following supporting planes
$$
\left\{l_1^{(i)} = l_1 + \delta^{(i)}, l_2^{(i)} = l_2, \cdots, l_s^{(i)} = l_s \right\}
$$ 
where $\delta^{(i)}$ is a real number. Assume that $\delta^{(i)}$ monotonically converges to $0$. Then, the tropicalization $\Trop(Q^{(i)})$ converges to the tropicalization $\Trop (Q)$ with respect to the Hausdorff distance.
\end{lemma}

\begin{proof}
We assume that $\delta^{(i)}$ is monotonically decreasing because we can similary deal with the case where the sequence is monotonically increasing. To present the proof in an organized manner, we begin by stating and proving two sublemmas. 

As the direction $\v_1$ is different from that of $\v_j$ for $j \geq 2$, we can take and fix a direction $\v_{1j}$ such that for all $t > 0$ and every $\u \in \R^n$, 
\begin{equation}\label{equationlemma7.311}
l_1(\u + t \v_{1j}) < l_1 (\u), \quad l_j(\u + t \v_{1j}) > l_j (\u).
\end{equation}
Regardless of our choice of $\u \in \R^n$, we have 
\begin{equation}\label{equationlemma7.312}
l_j (\u + t \v_{1j}) - l_j (\u) = c_{1j} \cdot t
\end{equation}
where $c_{1j} := \langle \v_j, \v_{1j} \rangle$, which is positive. Since $\delta^{(i)} \to 0$ as $i \to \infty$, for any $\varepsilon > 0$, there exists an integer $\iota(\varepsilon)$ such that for all $i \geq \iota(\varepsilon)$, $\delta^{(i)}$ is less than the minimum of the values $c_{1j} \cdot (\varepsilon/2)$ for $j$ with $2 \leq j \leq s$. 

\begin{sublemma}\label{sublemma1}
For every $\u_1 \in \Trop(Q)$ and the above choice of $\iota(\varepsilon)$, we have 
$$
\inf_{\u_2 \in \Trop(Q^{(i)})} d(\u_1, \u_2) < \varepsilon
$$
whenever $i \geq \iota(\varepsilon)$. Here, we emphasize that $\iota(\varepsilon)$ is independent of $\u_1$.
\end{sublemma}

\begin{proof} (of Sublemma~\ref{sublemma1})
For $\u_1 \in \Trop(Q)$, by Proposition~\ref{explicitdescriptionoftropicalizationoffunction}, there exists at least two indices $j_1$ and $j_2$  with $j_1 < j_2$ such that $l_{j_1}(\u_1) = l_{j_2}(\u_1)$ is the minimum of $\{l_{j}(\u_1) : 1 \leq j \leq s \}$. If one can take such indices $j_1$ and $j_2$ bigger than $1$, then $\u_1$ is contained in $\Trop(Q^{(i)})$ because the number $l_{j_1}(\u_1) = l_{j_1}^{(i)}(\u_1) = l_{j_2}^{(i)}(\u_1) = l_{j_2}(\u_1)$ is still the minimum of $\{ l_{j}^{(i)}(\u_1) : j = 1, \cdots ,s \}$ and hence $\inf_{\u_2 \in \Trop(Q^{(i)})} d(\u_1, \u_2) = 0 < \varepsilon$. Thus, it remains to consider the case where $l_1(\u_1) = l_2(\u_1) < l_j(\u_1)$ for any $j \geq 3$. Here, $j_2$ is assumed to be $2$ by renumbering $\{j: 2 \leq j \leq s\}$ without shuffling $j=1$. In this case, $\u_1$ is not in $\Trop(Q^{(i)})$ anymore because the minimum $l_2^{(i)}(\u_1) = l_2(\u_1)$ is solely attained by $l_2$. Taking $\u := \u_1 + (\varepsilon / 2) \v_{12} \in B(\u_0, \varepsilon)$, by ~\eqref{equationlemma7.311} and ~\eqref{equationlemma7.312}, we obtain 
\begin{align*}
l_1^{(i)}(\u) - l_2 (\u) &< l_1^{(i)}(\u_1) - l_2 (\u_1) - c \cdot (\varepsilon / 2) \\
&= l_1^{(i)}(\u_1) - l_1 (\u_1) - c \cdot (\varepsilon / 2) = \delta^{(i)} - c \cdot (\varepsilon / 2)
\end{align*}
for the positive constant $c := \min \{c_{1j}: 2 \leq j \leq s \}$. By our choice of $\iota(\varepsilon)$, whenever $i \geq \iota(\varepsilon)$, $\delta^{(i)} - c \cdot (\varepsilon / 2) \leq 0$ and hence $l^{(i)}_1(\u) \leq l_2(\u) = l^{(i)}_2(\u)$. 
Since $l^{(i)}_1(\u_1) > l^{(i)}_2(\u_1)$ and $l^{(i)}_1(\u) \leq l_2(\u) = l^{(i)}_2(\u)$ for some $\u \in B(\u_1, \varepsilon)$, it implies that there exists a point of $\Trop(Q^{(i)})$ contained in $B(\u_1, \varepsilon)$ if $i \geq \iota(\varepsilon)$. Therefore, the sublemma is justifed.
\end{proof}

As the direction of $\v_1$ is different from that of $\v_j$ for $j \geq 2$, we can take and fix a direction $\v_{1j}$ such that for all $t > 0$,
\begin{equation}\label{equationlemma7.33}
l_1 (\u + t \v_{1j}) > l_1 (\u), \quad l_j (\u + t \v_{1j} ) < l_j (\u) 
\end{equation}
for any $\u \in \R^n$. Regardless of our choice of $\u$, we have 
\begin{equation}\label{equationlemma7.34}
l_1 (\u + t \v_{1j} ) - l_1 (\u) = c_{1j} \cdot t
\end{equation}
where $c_{1j} := \langle \v_1, \v_j \rangle$, which is positive. Since $\delta^{(i)} \to 0$ as $i \to \infty$, for given $\varepsilon > 0$, there exists an integer $\nu(\varepsilon)$ such that for all $i$ with $i \geq \nu(\varepsilon)$, $\delta^{(i)}$ is less than the minimum of the values $c_{1j} \cdot (\varepsilon / 2)$ for $j$ with $2 \leq j \leq s$.

\begin{sublemma}\label{sublemma2}
For any sequence $\{ \u^{(i)}_2 \in \Trop (Q^{(i)}) : i \geq 1 \}$ and the above choice $\nu(\varepsilon)$, we have 
$$
\inf_{\u_1 \in \Trop(Q)} d(\u_1, \u_2^{(i)}) < \varepsilon
$$
whenever $i \geq \nu(\varepsilon)$. Here, we emphasize that $\nu(\varepsilon)$ is independent of $\{ \u^{(i)}_2 \in \Trop (Q^{(i)}) : i \geq 1 \}$.
\end{sublemma}

\begin{proof} (of Sublemma~\ref{sublemma2})
For $\u_2^{(i)} \in \Trop(Q^{(i)})$, by Proposition~\ref{explicitdescriptionoftropicalizationoffunction}, there exist at least two indices $j_1$ and $j_2$ with $1 \leq j_1 < j_2$ such that $l_{j_1}^{(i)} (\u_2^{(i)}) = l_{j_2}^{(i)} (\u_2^{(i)})$ is the minimum of $\{l_j^{(i)} (\u_2^{(i)}) : 1 \leq j \leq s \}$. As soon as $l_1(\u_2^{(i)}) = l_1^{(i)}(\u_2^{(i)}) - \delta^{(i)} \geq l_{j_1}^{(i)} (\u_2^{(i)}) = l_{j_2}^{(i)} (\u_2^{(i)}) = l_{j_1} (\u_2^{(i)}) = l_{j_2} (\u_2^{(i)})$, the minimum of $\{l_j (\u_2^{(i)}) : 1 \leq j \leq s \}$ is attained by $l_{j_1} (\u_2^{(i)})= l_{j_2}(\u_2^{(i)})$ and thus $\u_2^{(i)} \in \Trop(Q)$ and hence $\inf_{\u_1 \in \Trop(Q)} d(\u_1, \u_2^{(i)}) = 0 < \varepsilon$. Thus, it suffices to consider the case where 
\begin{equation}\label{equationlemma7.31}
l_1(\u_2^{(i)}) = l_1^{(i)}(\u_2^{(i)}) - \delta^{(i)} < l_{j}^{(i)} (\u_2^{(i)}) = l_{j} (\u_2^{(i)}).
\end{equation} 
for all $j \geq 2$. Since we have two indices $j_1, j_2$ with $l_{j_1}^{(i)} (\u_2^{(i)}) = l_{j_2}^{(i)} (\u_2^{(i)})$, we may assume that $l_2^{(i)}(\u_2^{(i)}) $ is the minimum of $\{ l_j^{(i)}(\u^{(i)}_2) : 1 \leq l \leq s\}$ by renumering $\{ 2 \leq j \leq s\}$ if necessary (Note that $j=1$ is fixed). In particular, we have 
\begin{equation}\label{equationlemma7.32}
l_2(\u^{(i)}_2) = l_2^{(i)}(\u^{(i)}_2) \leq l_1^{(i)}(\u^{(i)}_2).
\end{equation}
Taking $\u^{(i)} := \u^{(i)}_2 + (\varepsilon/2) \v_{12}$, by ~\eqref{equationlemma7.33}, ~\eqref{equationlemma7.34} and ~\eqref{equationlemma7.32}, we obtain
\begin{align*}
l_{2}(\u^{(i)}) - l_1(\u^{(i)}) &< l_{2}(\u_2^{(i)}) - l_1(\u_2^{(i)}) - c \cdot (\varepsilon / 2) \\
&\leq l_1^{(i)}(\u_2^{(i)}) - l_1(\u_2^{(i)}) - c \cdot (\varepsilon/2) = \delta^{(i)} - c \cdot (\varepsilon/2)
\end{align*}
for the positive constant $c := \min \{c_{1j}: 2 \leq j \leq s \}$. By our choice of $\nu(\varepsilon)$, whenever $i \geq \nu(\varepsilon)$, $\delta^{(i)} - c \cdot (\varepsilon / 2) \leq 0$ and $l_2(\u^{(i)}) \leq l_1(\u^{(i)})$. Since $l_1(\u^{(i)}_2) < l_2(\u^{(i)}_2)$ from ~\eqref{equationlemma7.31} and $l_2(\u^{(i)}) \leq l_1(\u^{(i)})$ for some $\u^{(i)} \in B(\u_2^{(i)}, \varepsilon)$, it implies that there exists a point of $\Trop(Q)$ contained in $B(\u^{(i)}_2, \varepsilon)$ if $i \geq \nu(\varepsilon)$. Therefore, the sublemma is justifed. 
\end{proof}

By Sublemma~\ref{sublemma1} and Sublemma~\ref{sublemma2}, for $i \geq \max(\iota(\varepsilon), \nu(\varepsilon))$, we obtain
\begin{align*}
&\sup_{\u_1 \in \Trop(Q)} \left( \inf_{\u_2 \in \Trop(Q^{(i)})} d(\u_1, \u_2) \right) < \varepsilon \\
&\sup_{\u_2 \in \Trop(Q^{(i)})} \left( \inf_{\u_1 \in \Trop(Q)} d(\u_1, \u_2) \right) < \varepsilon,
\end{align*}
which yields that 
$$
d_\text{Haus} \left(\Trop(Q), \Trop(Q^{(i)})\right) < \varepsilon.
$$
This completes the proof of Lemma~\ref{Hasudorffdistanceoftropicalization}.

\end{proof}

Now, we are ready to prove Theorem~\ref{maintheorem2}. 

\begin{proof} (of Theorem~\ref{maintheorem2})
Suppose that $L(\u)$ is a bulk-balanced fiber in a compact toric symplectic manifold $X$. Then, there exist a sequence $\{ P^{(i)} \}$ of polytopes in $M_\R \simeq \R^n$ and a sequence $\{ \u^{(i)}:  \u^{(i)} \in \text{Int}(P^{(i)}) \}$ of positions such that $P^{(i)}$ converges to $P$ with respect to the Hausdorff distance and $\u^{(i)}$ converges to an interior point $\u$ of $P$ with respect to the Euclidean distance. In order to show that $L(\u)$ is a strongly bulk-balanced fiber, it is enough to show that $\u$ lies in $\Trop(P, \m)$ for any $\m \in M$ by Thoerem~\ref{THEOREMA}. 

Since the fan $\Sigma$ of the toric manifold $X$ is fixed in the sequence of polytopes, each $P^{(i)}$ can be constructed by translating facets of $P$. Namely, 
$$
P^{(i)} = \{ \u : l^{(i)}_1(\u) \geq 0, \cdots , l^{(i)}_m(\u) \geq 0 \}
$$
where $l_j^{(i)}(\u) := l_j(\u) + \delta_j^{(i)}$. By Lemma~\ref{convergenceoflambdaj}, for each $j$, $\delta^{(i)}_j$ converges to $0$ as $i \to \infty$. By taking a subsequence if necessary, we may assume that for each $j$, the sequence $\{ \delta^{(i)}_j : i = 1, 2, \cdots \}$ is monotonic. 

We claim that $\Trop(P^{(i)}, \m)$ converges to $\Trop(P, \m)$. By renumbering the index, we may assume
$$
\left\{ l_j : \langle \m, \v_j \rangle \neq 0 \right\} = \left\{ l_1, \cdots, l_s \right\}.
$$
Let 
\begin{align*}
Q_j^{(i)} := \{ \u: l_{1}^{(i)}(\u) \geq 0, \cdots, l_{j}^{(i)}(\u) \geq 0, \, \l_{j+1}(\u) \geq 0, \cdots,   l_{s}(\u) \geq 0 \}.
\end{align*}
Note that $\Trop(Q_0^{(i)}) = \Trop (P, \m)$ and $\Trop(Q_s^{(i)}) = \Trop (P^{(i)}, \m)$. By the triangle inequality, we have 
\begin{align*}
d_{\text{Haus}} &(\Trop(P, \m), \Trop(P^{(i)}, \m)) = d_{\text{Haus}} (\Trop(Q_0^{(i)}), \Trop(Q_s^{(i)})) \\
&\leq d_{\text{Haus}} (\Trop(Q_0^{(i)}), \Trop(Q_1^{(i)})) + \cdots + d_{\text{Haus}} (\Trop(Q_{s-1}^{(i)}), \Trop(Q_s^{(i)})) 
\end{align*}
By Lemma~\ref{Hasudorffdistanceoftropicalization}, the right-hand side converges to $0$. Hence, the claim is derived. 

As $\u^{(i)} \in \Trop(P^{(i)}, \m)$ converges to $\u \in \text{Int}(P)$ and $\Trop(P^{(i)}, \m)$ converges to  $\Trop(P, \m)$, $\u$ is also contained in $\Trop(P, \m)$. Hence, the proof is completed.
\end{proof}

\end{document}